\documentclass[11pt]{article}
\usepackage{amsfonts}
\usepackage{amssymb}
\usepackage{amsthm}
\usepackage{amsmath}
\usepackage{bm}
\usepackage{amscd} 

\usepackage{tikz}
\usetikzlibrary{matrix}

\usepackage{hyperref}

\numberwithin{equation}{section} \DeclareMathSizes{2}{10}{12}{13}
\makeatletter
\newcommand*{\doublerightleftarrow}[2]{\mathrel{
  \settowidth{\@tempdima}{$\scriptstyle#1$}
  \settowidth{\@tempdimb}{$\scriptstyle#2$}
  \ifdim\@tempdimb>\@tempdima \@tempdima=\@tempdimb\fi
  \mathop{\vcenter{
    \offinterlineskip\ialign{\hbox to\dimexpr\@tempdima+1em{##}\cr
    \rightarrowfill\cr\noalign{\kern.5ex}
    \leftarrowfill\cr}}}\limits^{\!#1}_{\!#2}}}
\makeatother

\newtheorem{thm}{Proposition}[section]

\newtheorem{Thm}[thm]{Theorem}

\newtheorem{lem}[thm]{Lemma}
\newtheorem{defn}[thm]{Definition}

\title{On Noetherian schemes over $(\mathcal C,\otimes,1)$ and the category of quasi-coherent sheaves}
\author{Abhishek Banerjee}

\date{}

\begin{document}

\maketitle

\centerline{\small \emph{Institut des Hautes \'{E}tudes Scientifiques, Le Bois-Marie 35, route de Chartres, }}
\centerline{\small \emph{91440, Bures-sur-Yvette, France. Email: abhishekbanerjee1313@gmail.com.}}

\medskip
\begin{abstract} Let $(\mathcal C,\otimes,1)$ be an abelian symmetric monoidal category satisfying certain conditions
and let $X$ be a  scheme over $(\mathcal C,\otimes,1)$ in the sense of To\"{e}n and Vaqui\'{e}. In this paper we show 
that when $X$ is quasi-compact and semi-separated,  any quasi-coherent sheaf on $X$ may be expressed as a directed
colimit of its finitely generated quasi-coherent submodules. Thereafter, we introduce a notion of ``field objects'' in
$(\mathcal C,\otimes,1)$ that satisfy several properties similar to those of fields in usual commutative algebra. Finally
we show that the points of a Noetherian, quasi-compact and semi-separated scheme $X$ over such a field
object $K$ in $(\mathcal C,\otimes,1)$ can be recovered from  certain kinds of functors between categories
of quasi-coherent sheaves. The latter is a partial generalization of some recent results of Brandenburg and
Chirvasitu. 
\end{abstract}

\medskip
{\bf MSC(2010) Subject classification: 14A15, 19D23.}

\medskip

\section{Introduction}

\medskip

\medskip
Let $(\mathcal C,\otimes,1)$ be an abelian symmetric monoidal category satisfying certain conditions. Then, the idea
of doing algebraic geometry over the category $\mathcal C$ has been developed by several authors; see, for instance,
Deligne \cite{Del}, Hakim \cite{Hak} and the work of To\"{e}n and Vaqui\'{e} \cite{TV}. When $\mathcal C=k-Mod$, the category
of modules over an ordinary commutative ring $k$, we recover the usual algebraic geometry of schemes 
over $Spec(k)$. In general, a more abstract theory of schemes 
(and monoid objects) using 
categories is the entry point to, for instance, the derived algebraic
geometry of Lurie \cite{Lurie} and the homotopical algebraic
geometry of To\"{e}n and Vezzosi (see \cite{TVz1}, \cite{TVz2}). For an 
abstract treatment of monoid objects in abelian model categories, 
see, for instance Hovey \cite{Hoveyp}. Further, the Morita theory
for monoids in symmetric monoidal categories has been developed
by Vitale \cite{Vitale}. The algebraic geometry over symmetric monoidal categories is also
a stepping stone to the study of schemes over ``the field with one element''
$\mathbb F_1$ (for more on the geometry over $\mathbb F_1$, we refer the reader for example to the work of Connes and Consani \cite{CC1}, \cite{CC2}, \cite{ACKC}, Deitmar \cite{Die} and Soul\'{e} \cite{Soule}, \cite{Soule2}). In
the last few years, there has also been a lot of interest in the theory
of monoid schemes (see Corti\~{n}as, Haesemeyer, Walker and Weibel \cite{CHWW},
Flores and Weibel \cite{FW}, Pirashvili \cite{IP} and Vezzani \cite{Vezzani}).

\medskip
In this paper, we will work with the notion of a scheme over $(\mathcal C,\otimes,1)$
due to To\"{e}n and Vaqui\'{e} \cite{TV}. For most of this article, we will also assume that the abelian symmetric
monoidal category $\mathcal C$ is also ``locally finitely generated''.  The theory of   locally finitely generated abelian  categories   has been studied extensively in the literature 
(see, for example,  \cite{Gar}, \cite{Prest}, \cite{Prest3}, \cite{Prest2}, \cite{Sten}).  For a scheme $X$ over
$(\mathcal C,\otimes,1)$, the purpose
of this paper is to consider the following two questions on the category $QCoh(X)$ of quasi-coherent sheaves on $X$:

\medskip
(Q1) For a scheme $X$ over $(\mathcal C,\otimes,1)$, can a quasi-coherent sheaf on $X$ be expressed as a colimit of its finitely generated quasi-coherent submodules? 

\medskip We show that when $X$ 
is quasi-compact and semi-separated (in the sense of Definition \ref{D2.2}),  this is indeed true, i.e., any quasi-coherent sheaf on $X$ can be expressed as a directed colimit of finitely generated
quasi-coherent submodules (see Theorem \ref{Thm3.8}). For ordinary schemes over
a field, similar results have been studied in the classical texts (see EGA I \cite{EGA1}, \cite{EGA1b}). More recently, Rydh \cite{Rydh}, \cite{Rydh2} has tackled
similar questions for algebraic stacks.

\medskip
The second question we consider in this paper concerns the points of a scheme over
a ``field object'' $K$ in $(\mathcal C,\otimes,1)$. In Definition \ref{D4.2}, we have introduced a notion of ``field objects'' in the symmetric monoidal category
$(\mathcal C,\otimes,1)$ that we believe is of independent interest. It is shown that for a field object $K$ 
over $(\mathcal C,\otimes,1)$, the category
$K-Mod$ of $K$-modules   satisfies several properties analogous to the category of vector spaces over a field. Thereafter, we ask the following question: 

\medskip 
(Q2) For a scheme $X$ over $(\mathcal C,\otimes,1)$, under what conditions does
a functor $F:QCoh(X)\longrightarrow QCoh(Spec(K))=K-Mod$ to the category
 of $K$-modules correspond to a pullback $F\cong f^*$ by a morphism
$f:Spec(K)\longrightarrow X$?

\medskip
For a quasi-compact, semi-separated and 
Noetherian scheme $X$, we show that any cocontinuous, symmetric monoidal and normal functor $F:QCoh(X)
\longrightarrow K-Mod$ corresponds to a point of $X$ over the field object $K$ (see Theorems  \ref{P5.13wvo} and
\ref{finres}). The notion of a normal
functor $F:QCoh(X)
\longrightarrow K-Mod$ introduced in 
Definition \ref{NDef}  from the category of quasi-coherent sheaves  is an extension of the notion of normal functors
between categories of modules in \cite{Vitale}.

\medskip
We now describe the structure of the paper in greater detail. We let $Comm(\mathcal C)$ be the category of unital
commutative monoid objects in $\mathcal C$. For any $A\in Comm(\mathcal C)$, we let $A-Mod$ be the category
of $A$-modules in $\mathcal C$. Then, we let $Aff_{\mathcal C}:=Comm(\mathcal C)^{op}$
be the category of affine schemes over $\mathcal C$. In particular, the affine scheme corresponding to a commutative
monoid object $A\in Comm(\mathcal C)$ is denoted by $Spec(A)$.  In Section 2, we briefly recall the notion of a scheme
over $(\mathcal C,\otimes,1)$ due to To\"{e}n and Vaqui\'{e} \cite{TV}. We also recall from \cite{AB1} and \cite{AB2}
the notion of quasi-coherent sheaf for schemes over $(\mathcal C,\otimes,1)$ as well as the corresponding formalism of pullback
and pushforward functors. Thereafter, in Section 3, we assume that every object in $\mathcal C$ can be expressed
as a directed colimit of finitely generated subobjects. We show that this is equivalent to every $A$-module $M$ being
isomorphic to the directed colimit of its finitely generated $A$-submodules.  Then, the main result of Section 3 is the following (see Theorem \ref{Thm3.8}).

\begin{Thm} Let $X$ be a quasi-compact and semi-separated scheme over $(\mathcal C,\otimes,1)$ and let 
$\mathcal M$ be a quasi-coherent sheaf on $X$. Then, $\mathcal M$ can be expressed as a filtered direct limit
of its finitely generated quasi-coherent submodules. 
\end{Thm}

\medskip
We start considering Noetherian schemes in Section 4. We say that a commutative monoid object $A$ is Noetherian if
every finitely generated $A$-module $M$ is finitely presented, i.e., can be expressed as a colimit
\begin{equation}
M\cong colim(0\longleftarrow A^m\overset{q}{\longrightarrow}A^n)
\end{equation} for some morphism $q:A^m\longrightarrow A^n$ with $m$, $n\geq 1$ (see Definition \ref{Def4.1}). Then, our first
result is that if $A$ is a Noetherian commutative monoid object, $Spec(A)$ is a Noetherian scheme, i.e., if $A\longrightarrow B$
is a morphism in $Comm(\mathcal C)$ inducing a Zariski open immersion $Spec(B)\longrightarrow Spec(A)$ of affine schemes, 
$B\in Comm(\mathcal C)$ must also be Noetherian (see Proposition \ref{P4.2xr}). The key notion in Section 4 is that of a ``field object'': we say that a Noetherian commutative
monoid object $0\ne K\in Comm(\mathcal C)$ is a ``field object'' in $(\mathcal C,\otimes,1)$ if it has no subobjects other than
$0$ and $K$ in $K-Mod$ (see Definition \ref{D4.2}). We believe that this notion
is of independent interest and hope that it would be a first step towards developing Galois theory for schemes over a symmetric monoidal category. 
In order to justify our definition as the correct notion for a field in a symmetric monoidal category, the rest of Section 4 is devoted to showing that the category $K-Mod$ of modules over a field
object $K$ has several properties similar to the category of vector spaces over a field. We prove the following succession of results on $K-Mod$, each property being utilized to prove the next. 

\medskip
\begin{Thm} Let $K$ be a field object in $(\mathcal C,\otimes,1)$ in the sense of 
Definition \ref{D4.2}. Then, we show the following: 

\smallskip
(a) For any $n\geq 1$,  $K^n$ is a projective object of $K-Mod$  and any morphism of finitely presented $K$-modules can be lifted to a morphism
of their presentations. 

\smallskip
(b) Every monomorphism (resp. epimorphism) in $K-Mod$ is a split monomorphism 
(resp. split epimorphism).

\smallskip
(c) Any finitely generated (non-zero) $K$-module is  isomorphic to a direct
sum $K^m$ for some $m\geq 1$. 

\smallskip 
(d) The corresponding
affine scheme $Spec(K)$ behaves in a way similar to a space with a single point, i.e., any non-trivial Zariski open immersion $U\longrightarrow Spec(K)$
must be an isomorphism. 

\end{Thm}

\medskip Finally, given a Noetherian commutative monoid object $A$ such that $Hom_{A-Mod}(A,A)$ is an integral
domain, we describe in Proposition \ref{Final4} a process of localizing $A$ (under some conditions) to obtain a field object $K(A)$ in $(\mathcal C,\otimes,1)$. This is the analogue of the usual
contruction of the field of fractions of an integral domain. 

\medskip
In Section 5, we consider the points of a Noetherian, quasi-compact and semi-separated scheme $X$ over a field object $K$, i.e., 
morphisms $f:Spec(K)\longrightarrow X$. It is clear that any such morphism $f$ defines a pullback functor
$f^*:QCoh(X)\longrightarrow QCoh(Spec(K))=K-Mod$ that is cocontinuous (i.e., preserves small colimits) and preserves
the symmetric monoidal structure. As such, it is natural to ask if the converse is true, i.e., whether any cocontinuous
symmetric monoidal functor $F:QCoh(X)\longrightarrow K-Mod$ can be described as $F\cong f^*$ for some
morphism $f:Spec(K)\longrightarrow X$. For usual schemes, Brandenburg and Chirvasitu \cite{BC} have shown that this is indeed the case. In fact,
it is shown in \cite{BC} that any cocontinuous symmetric monoidal functor $F:QCoh(X)\longrightarrow QCoh(Y)$ with $X$ being a quasi-compact
and quasi-separated scheme can be described as a pullback   along some morphism $f:Y\longrightarrow X$. Let  $A$, $B$ be commutative
monoid objects of $(\mathcal C,\otimes,1)$ with $A$ Noetherian. Then, if  $F:A-Mod
\longrightarrow B-Mod$ is a functor that is not only cocontinuous and symmetric monoidal but also normal in the sense of 
Vitale \cite[$\S$ 4]{Vitale}, we show in Proposition \ref{TYpcf} that $F$ corresponds to ``extension of scalars'' along some morphism $A\longrightarrow B$
in $Comm(\mathcal C)$. This suggests that we should introduce a suitable notion of normal functor from
$QCoh(X)$ to $K-Mod$. This is done in Definition \ref{NDef}. Then, the main result of 
Section 5 is the following (see Theorems \ref{P5.13wvo} and \ref{finres}).

\begin{Thm} Let $X$ be a quasi-compact, semi-separated and Noetherian scheme over $(\mathcal C,\otimes,1)$. 
Let $K$ be a field object of $(\mathcal C,\otimes,1)$ and $F:QCoh(X)\longrightarrow K-Mod$ be a cocontinuous
symmetric monoidal functor that is also normal. Then, there exists a morphism $f:Spec(K)\longrightarrow X$  such that 
$F\cong f^*$.

\smallskip
Conversely, the pullback functor 
$f^*:QCoh(X)\longrightarrow K-Mod$ corresponding to a morphism $f:Spec(K)
\longrightarrow X$ is not only cocontinuous and symmetric monoidal, but also a normal functor in the sense
of Definition \ref{NDef}.  
\end{Thm}

\medskip
As indicated above, from Section 3 onwards, we will assume that the   category $(\mathcal C,\otimes,1)$ is ``locally finitely generated''. As such, we present here some natural examples of situations where this condition applies. 

\medskip 
\noindent {\bf Examples:}  (a) If $Y$ is a topological space and $\mathcal A$ is a presheaf of commutative rings on $Y$,
we could take $\mathcal C$ to be the category $\mathcal A-Premod$ of presheaves of $\mathcal A$-modules on $Y$ (see \cite[Corollary 2.15]{Prest3}).  

\medskip
(b) Further, if $Y$ is any topological space with a basis of compact open sets (for example,
any locally  Noetherian space) 
and $\mathcal A$ is any sheaf of commutative rings on $Y$, the category $\mathcal A-Mod$ of sheaves of
$\mathcal A$-modules on $Y$ is also locally finitely generated (see \cite[Theorem 3.5]{Prest}).  In fact, in cases (a) and (b), 
the categories $\mathcal A-Premod$ and $\mathcal A-Mod$ respectively satisfy an even  stronger
condition, i.e., they are actually ``locally finitely presented'' (see \cite{Prest}, \cite{Prest3}). 

\medskip
(c)  If $Y$ is the closed unit interval $[0,1]$ with
the usual topology and $\mathcal A_Y$ is the sheaf of continuous real-valued functions on $Y$, the category 
$\mathcal A_Y-Mod$ of sheaves of $\mathcal A_Y$-modules on $Y$ is also locally finitely generated (see \cite[Proposition 5.5]{Prest}).

\medskip
In this paper, given any symmetric monoidal categories $(C,\otimes,1_C)$ and $(D,\otimes,1_D)$,  a functor
$F:C\longrightarrow D$ will be said to be symmetric monoidal if it preserves the symmetric monoidal structures
up to canonical isomorphism. Additionally, all the symmetric monoidal categories in this paper will
be $\mathbb Z$-linear (i.e., preadditive) and therefore we will only speak of symmetric monoidal functors between them
that are also $\mathbb Z$-linear.  

\medskip

\medskip

\section{Quasi-coherent sheaves on schemes over $(\mathcal C,\otimes,1)$}

\medskip

\medskip
In this section as well as in the rest of this paper, we let $(\mathcal C,\otimes,1)$ denote an abelian symmetric monoidal 
category that contains all small limits and colimits. Further, we will assume that $(\mathcal C,\otimes,1)$ is closed; in other words, given any objects $X$, $Y\in \mathcal C$, there
exists an internal hom object $\underline{Hom}(X,Y)\in \mathcal C$ such that we have natural isomorphisms:
\begin{equation}
Hom(Z\otimes X,Y)\cong Hom(Z,\underline{Hom}(X,Y))\qquad\forall\textrm{ }Z\in \mathcal C
\end{equation} Let $Comm(\mathcal C)$ denote the category of commutative monoid objects in $\mathcal C$. If $A\in Comm(\mathcal C)$ is such
a commutative monoid, we let $A-Mod$ denote the category of $A$-module objects in $\mathcal C$. Then, it follows that
$(A-Mod,\otimes_A,A)$ is also a closed abelian symmetric monoidal category (see Vitale \cite{Vitale}). For any $M$, $N\in A-Mod$, we will denote
by $\underline{Hom}_A(M,N)$ the internal hom object in $A-Mod$. Additionally, we will assume
that filtered colimits commute with finite limits in each $A-Mod$. 

\medskip
Let $Aff_{\mathcal C}:=Comm(\mathcal C)^{op}$ be the category of affine schemes over $\mathcal C$. Given a commutative
monoid object $A\in Comm(\mathcal C)$, we let $Spec(A)$ denote the affine scheme corresponding to $A$. Then, To\"{e}n and 
Vaqui\'{e} \cite[D\'{e}finition 2.10]{TV} have introduced the notion of Zariski coverings in $Aff_{\mathcal C}$ making 
$Aff_{\mathcal C}$ into a subcanonical Grothendieck site. Accordingly, we let $Sh(Aff_{\mathcal C})$ be the category of sheaves
of sets on $Aff_{\mathcal C}$. Further, To\"{e}n and Vaqui\'{e} have also introduced a  notion of Zariski open immersions
in the category $Sh(Aff_{\mathcal C})$ (see \cite[D\'{e}finition 2.12]{TV}) that is stable under composition and base change. We now recall
from \cite[D\'{e}finition 2.15]{TV} the notion of a scheme over the symmetric monoidal category $(\mathcal C,\otimes,1)$. 

\medskip
\begin{defn} Let $X$ be an object of $Sh(Aff_{\mathcal C})$. Then, $X$ is said to be a scheme over $(\mathcal C,\otimes,1)$ if there
exists a family $\{X_i\}_{i\in I}$ of affine schemes over $\mathcal C$ along with a morphism:
\begin{equation}\label{2.2} 
p:\coprod_{i\in I} X_i\longrightarrow X 
\end{equation} satisfying the following two conditions:

(1) The morphism $p$ is an epimorphism in the category $Sh(Aff_{\mathcal C})$.

(2) For each $i\in I$, the morphism $X_i\longrightarrow X$ is a Zariski open immersion in the category $Sh(Aff_{\mathcal C})$. 
\end{defn}

\medskip
A collection of morphisms $\{X_i\longrightarrow X\}_{i\in I}$ as in \eqref{2.2} is said to be an affine cover of the scheme $X$. Given a scheme
$X$ over $\mathcal C$, we denote by $ZarAff(X)$ the category of Zariski open immersions $U\longrightarrow X$ with $U$ affine. By definition (see \cite[D\'{e}finition 2.10]{TV}), a 
collection  $\{U_i=Spec(A_i)\longrightarrow Spec(A)\}_{i\in I}$  of Zariski open immersions is said to be a covering of $Spec(A)$ if there exists a finite
subset $J\subseteq I$ such that a morphism $f:M\longrightarrow N$ in $A-Mod$ is an isomorphism if and only if the induced morphism
$f_j:=f\otimes_AA_j:M\otimes_AA_j\longrightarrow N\otimes_AA_j$ is an isomorphism for each $j\in J$. 

\medskip
\begin{defn}\label{D2.2}  Let $X$ be a scheme over $(\mathcal C,\otimes,1)$. 
 We will say that   $X$  is quasi-compact if every affine cover of $X$ has a finite subcover. 
 
 \medskip Further, we will say that $X$ is semi-separated if the fiber product $U\times_XV\in ZarAff(X)$ for any $U$, $V\in ZarAff(X)$. 
\end{defn} 

\medskip
From the above, it is clear that for any commutative monoid object $A$, the affine scheme $Spec(A)$ is quasi-compact and semi-separated. Further, a scheme
$X$ over $\mathcal C$ determines a functor:
\begin{equation}\label{2.3pje}
\mathcal O_X:ZarAff(X)^{op}\longrightarrow Comm(\mathcal C) \qquad (U=Spec(A)\longrightarrow X)\mapsto A
\end{equation} In \cite{AB1}, we introduced the notion of a quasi-coherent sheaf on a scheme $X$ over $\mathcal C$. For this, we consider the category
$Mod_{\mathcal C}$ whose objects are pairs $(A,M)$ where $A\in Comm(\mathcal C)$ and $M\in A-Mod$. A morphism $(f,f_\sharp):(A,M)
\longrightarrow (B,N)$ in $Mod_{\mathcal C}$ consists of a morphism $f:A\longrightarrow B$ in $Comm(\mathcal C)$ along with
a morphism $f_\sharp:B\otimes_AM\longrightarrow N$ of $B$-modules. 

\medskip
\begin{defn} (see \cite[Definition 2.2]{AB1}) Let $X$ be a scheme over $(\mathcal C,\otimes,1)$. A quasi-coherent sheaf $\mathcal M$ on $X$
is a functor 
\begin{equation}
\mathcal M:ZarAff(X)^{op}\longrightarrow Mod_{\mathcal C}
\end{equation} satisfying the following two conditions: 

(1) If $p$ denotes the obvious projection $p:Mod_{\mathcal C}\longrightarrow Comm(\mathcal C)$, we have $p\circ \mathcal M=\mathcal O_X$. 

(2) For any morphism $u:U'\longrightarrow U$ in $ZarAff(X)$, let $\mathcal M(U)=(\mathcal O_X(U),M)$ and $\mathcal M(U')=
(\mathcal O_X(U'),M')$. Then, the morphism $\mathcal M(u)_\sharp:\mathcal O_X(U')\otimes_{\mathcal O_X(U)}M\longrightarrow 
M'$ is an isomorphism. 

\medskip
When there is no danger of confusion, for any $U\in ZarAff(X)$,
we will often use $\mathcal M(U)$ simply to denote the $\mathcal O_X(U)$-module corresponding to $\mathcal M(U)\in Mod_{\mathcal C}$. 
The category of quasi-coherent sheaves over $X$ will be denoted by $QCoh(X)$. 

\end{defn}

\medskip Given a scheme $X$, it is easy to see that $QCoh(X)$ becomes a symmetric monoidal category by setting:
\begin{equation}
(\mathcal M\otimes_{\mathcal O_X}\mathcal N)(U):=\mathcal M(U)\otimes_{\mathcal O_X(U)}\mathcal N(U) \qquad
\forall\textrm{ }U\in ZarAff(X)
\end{equation} for every $\mathcal M$, $\mathcal N\in QCoh(X)$. Let $f:X\longrightarrow Y$ be a morphism of schemes over
$(\mathcal C,\otimes,1)$. We will say that $f$ is quasi-compact if for each $V\in ZarAff(Y)$ and  $U:=V\times_YX$, every 
affine covering $\{U_j\longrightarrow U\}_{j\in J}$ of $U$ has a finite subcover. Further, we will say that $f:X\longrightarrow Y$
is semi-separated if for each $V\in ZarAff(Y)$, $U:=V\times_YX$ and $U_1$, $U_2\in ZarAff(U)$, the fiber product
$U_1\times_UU_2\in ZarAff(U)$. 

\medskip
\begin{thm} \label{P2.4} Let $f:X\longrightarrow Y$ be a morphism of schemes over $(\mathcal C,\otimes,1)$. Then:

\medskip
(a) There exists a pullback functor $f^*:QCoh(Y)\longrightarrow QCoh(X)$ along with natural isomorphisms $f^*(\mathcal M)\otimes_{\mathcal O_X} f^*(\mathcal N)\cong
f^*(\mathcal M\otimes_{\mathcal O_Y}\mathcal N)$ for $\mathcal M$, $\mathcal N\in QCoh(Y)$. 

\medskip
(b) If $f$ is quasi-compact and semi-separated, then there exists a pushforward $f_*:QCoh(X)\longrightarrow QCoh(Y)$ that is right adjoint
to the pullback $f^*$. 

\medskip
(c) For any $\mathcal M'$, $\mathcal N'\in QCoh(X)$, there are natural morphisms $f_*(\mathcal M')\otimes_{\mathcal O_Y}
f_*(\mathcal N')\longrightarrow f_*(\mathcal M'\otimes_{\mathcal O_X}
\mathcal N')$.

\end{thm}

\begin{proof} (a) Let $\{Y_i\longrightarrow Y\}_{i\in I}$ be an affine covering of $Y$. For each
$i\in I$, set $X_i:=X\times_YY_i$ and let $\{U_{ij}\longrightarrow X_i\}_{j\in J_i}$ be an affine covering of $X_i$. Then, the 
 functor $f^*:QCoh(Y)\longrightarrow QCoh(X)$, as defined in   \cite[Proposition 2.6]{AB2} is determined by setting
 $f^*(\mathcal M)(U_{ij}):=\mathcal M(Y_i)\otimes_{\mathcal O_Y(Y_i)}\mathcal O_X(U_{ij})$ for each $j\in J_i$, $i\in I$.  
 It follows that:
 \begin{equation}
 \begin{array}{ll}
 f^*(\mathcal M\otimes_{\mathcal O_Y}\mathcal N)(U_{ij})& \cong (\mathcal M(Y_i)\otimes_{\mathcal O_Y(Y_i)}\mathcal O_X(U_{ij}))
 \otimes_{\mathcal O_X(U_{ij})}(\mathcal N(Y_i)\otimes_{\mathcal O_Y(Y_i)}\mathcal O_X(U_{ij})) \\ & \cong f^*(\mathcal M)(U_{ij})
 \otimes_{\mathcal O_X(U_{ij})}f^*(\mathcal N)(U_{ij})
 \end{array}
 \end{equation}
 
 \medskip
 Part (b) is already shown in  \cite[Proposition 2.4]{AB2} and \cite[Proposition 2.6]{AB2}. Finally, for (c), we consider the counit
 morphisms $f^*f_*(\mathcal M')\longrightarrow \mathcal M'$, $f^*f_*(\mathcal N')\longrightarrow \mathcal N'$ given by the adjoint pair
 $(f^*,f_*)$. Using (a), we know that $(f^*f_*(\mathcal M'))\otimes_{\mathcal O_X}(f^*f_*(\mathcal N'))\cong f^*(f_*(\mathcal M')
 \otimes_{\mathcal O_Y}f_*(\mathcal N'))$ and hence we have a morphism $f^*(f_*(\mathcal M')
 \otimes_{\mathcal O_Y}f_*(\mathcal N'))
\longrightarrow \mathcal M'\otimes_{\mathcal O_X}\mathcal N'$.  Again, since $f^*$ is left adjoint to $f_*$, this gives us  a natural morphism:
$f_*(\mathcal M')\otimes_{\mathcal O_Y}
f_*(\mathcal N')\longrightarrow f_*(\mathcal M'\otimes_{\mathcal O_X}
\mathcal N')$.

\end{proof}

\medskip
From Proposition \ref{P2.4}, it follows that, in the language of \cite[XI.2]{ML}, $f^*:QCoh(Y)\longrightarrow QCoh(X)$ is a ``strong
tensor functor'' between the symmetric monoidal categories 
$QCoh(Y)$ and $QCoh(X)$ whereas $f_*:QCoh(X)\longrightarrow QCoh(Y)$ is a ``lax tensor functor'' (if $f$ is quasi-compact and
semi-separated). Further, since $f^*$ is a left adjoint, it preserves colimits, in other words, $f^*:QCoh(Y)\longrightarrow QCoh(X)$ is a cocontinuous functor. In particular, if $X$ is a scheme over $(\mathcal C,\otimes,1)$ that is semi-separated in the sense of Definition 
\ref{D2.2}, we note that for any $U\in ZarAff(X)$, the Zariski open immersion $i:U\longrightarrow X$ is quasi-compact and semi-separated. 
We now have the following result. 

\medskip
\begin{thm}\label{P2.5qp} Let $X$ be a semi-separated scheme over $(\mathcal C,\otimes,1)$. Let $i:U\longrightarrow X$ be a Zariski open
immersion that is quasi-compact and semi-separated. Then, the counit natural transformation $i^*i_*\longrightarrow 1$  is an isomorphism of functors $i^*i_*\cong 1:QCoh(U)\longrightarrow QCoh(U)$. In particular, this is true for $(i:U\longrightarrow X)\in ZarAff(X)$ . 
\end{thm}

\begin{proof} We consider some $\mathcal M\in QCoh(U)$. Then, since $i$ is quasi-compact and semi-separated, 
we have $i_*\mathcal M\in QCoh(X)$. Then, as in the proof of Proposition \ref{P2.4}(a), $i^*i_*\mathcal M\in QCoh(U)$ is determined
by the modules:
\begin{equation}\label{Eq2.7}
(i^*i_*\mathcal M)(W):=(i_*\mathcal M)(V)\otimes_{\mathcal O_X(V)}\mathcal O_X(W) \qquad \forall \textrm{ }V\in ZarAff(X), \textrm{ }W\in ZarAff(V\times_XU)
\end{equation}  Since $i$ is quasi-compact, we can choose a finite affine cover $\{W_i\}_{i\in I}$ of $i^{-1}(V)=V\times_XU$. Since $i^{-1}(V)=V\times_XU\longrightarrow V\longrightarrow X$ is a Zariski open immersion and hence a monomorphism,
it follows that $(W_i\times_{i^{-1}(V)} W_{i'})=W_i\times_{V}W_{i'}=W_i\times_XW_{i'}$ for any $i$, $i'\in I$. Then since $X$ is semi-separated,
$(W_i\times_{i^{-1}(V)} W_{i'}) =W_i\times_XW_{i'}$ is affine. Then,
by definition of $i_*$ (see \cite[Proposition 2.4]{AB2}, we know that 
\begin{equation}\label{Eq2.8rt}(i_*\mathcal M)(V)=lim\left(\prod_{i\in I}\mathcal M(W_i) \overset{\longrightarrow}{\underset{\longrightarrow}{ }} \prod_{i,i'\in I}
\mathcal M(W_i\times_{V} W_{i'})\right)
\end{equation}  Combining with \eqref{Eq2.7}, it follows
that 
\begin{equation}\label{Eq2.9qp}
\begin{array}{l}
(i^*i_*\mathcal M)(W)\\ =lim\left(\prod_{i\in I}\mathcal M(W_i)\overset{\longrightarrow}{\underset{\longrightarrow}{ }}\prod_{i,i'\in I}
\mathcal M(W_i\times_{V} W_{i'})\right)
\otimes_{\mathcal O_X(V)}\mathcal O_X(W) \\
= lim\left(\prod_{i\in I}\mathcal M(W_i)\otimes_{\mathcal O_X(V)}\mathcal O_X(W) \overset{\longrightarrow}{\underset{\longrightarrow}{ }}\prod_{i,i'\in I}
\mathcal M(W_i\times_{V} W_{i'})\otimes_{\mathcal O_X(V)}\mathcal O_X(W) \right) \\
\end{array}
\end{equation}  Additionally, for each $i\in I$, we have:
\begin{equation}\label{2.10qp}
\begin{array}{ll}
\mathcal M(W_i)\otimes_{\mathcal O_X(V)}\mathcal O_X(W)&=\mathcal M(W_i)\otimes_{\mathcal O_X(W_i)}(\mathcal O_X(W_i)
\otimes_{\mathcal O_X(V)}\mathcal O_X(W))\\
&=\mathcal M(W_i)\otimes_{\mathcal O_X(W_i)}\mathcal O_X(W_i\times_VW)
=\mathcal M(W_i\times_VW)\\
\end{array}
\end{equation} Similarly, for any $i$, $i'\in I$, we have:
\begin{equation}\label{2.11qp}
\mathcal M(W_i\times_VW_{i'})\otimes_{\mathcal O_X(V)}\mathcal O_X(W)=\mathcal M(W_i\times_VW_{i'}\times_VW)
\end{equation} Combining \eqref{Eq2.9qp}, \eqref{2.10qp} and \eqref{2.11qp}, it follows that:
\begin{equation}\label{2.12qp}
(i^*i_*\mathcal M)(W)=lim\left(\prod_{i\in I}\mathcal M(W_i\times_VW) \overset{\longrightarrow}{\underset{\longrightarrow}{ }}\prod_{i,i'\in I}
\mathcal M(W_i\times_VW_{i'}\times_VW)\right)
\end{equation} Finally, since $\{W_i\times_VW\}_{i\in I}$ is an affine cover of $W$, it follows from \cite[Corollaire 2.11]{TV}
that the limit in \eqref{2.12qp} is isomorphic to $\mathcal M(W)$. Hence, $i^*i_*\mathcal M(W)\cong \mathcal M(W)$ and the result follows.
\end{proof}

\medskip

\medskip

\section{Colimits of finitely generated quasi-coherent submodules}

\medskip

\medskip

In the usual algebraic geometry of schemes over $Spec(\mathbb Z)$, it is a classical fact that a quasi-coherent sheaf on a noetherian
scheme $X$ is the union of its coherent subsheaves (see \cite[Corollaire 9.4.9]{EGA1}). For schemes that are quasi-compact and quasi-separated, every quasi-coherent
sheaf is a filtered direct colimit of  finitely presented $\mathcal O_X$-modules  (see \cite[$\S$ 6.9]{EGA1b}). Similar results are known for certain kinds of algebraic
stacks (see \cite[Proposition 15.4]{LMB} and \cite{Rydh}, \cite{Rydh2}).

\medskip
In  this section, we will develop similar results  for  quasi-coherent sheaves on schemes over $(\mathcal C,\otimes,1)$. Given
a commutative monoid object $A\in Comm(\mathcal C)$, we will say that an $A$-module $M$  is finitely generated if it is finitely generated
as an object of the category $A-Mod$. In other words, given any filtered system of monomorphisms $\{M_i\}_{i\in I}$ in $A-Mod$, we have
an isomorphism:
\begin{equation}
colim_{i\in I}Hom_{A-Mod}(M,M_i)\overset{\cong}{\longrightarrow} Hom_{A-Mod}(M,colim_{i\in I}M_i)
\end{equation} 
Let $X$ be a quasi-compact scheme over $(\mathcal C,\otimes,1)$. We will say that a  quasi-coherent sheaf $\mathcal M$ on $X$ is finitely generated if $\mathcal M(U)$ is finitely generated as an $\mathcal O_X(U)$-module for every $U\in ZarAff(X)$. 

\medskip
\begin{thm} \label{P2.6} (a) Let $A\in Comm(\mathcal C)$ be a commutative monoid object and let $M\in A-Mod$ be an $A$-module. Let $\{M_i\}_{i\in I}$ be
a family (finite or infinite) of submodules of $M$. Then, there exists a submodule $\sum_{i\in I}M_i$ of $M$ that is the ``sum'' of the family
of submodules $\{M_i\}_{i\in I}$. In other words, the submodule $\sum_{i\in I}M_i$ satisfies the following two conditions:

\medskip
(1) For each $i\in I$, $M_i$ is a submodule of $\sum_{i\in I}M_i$.

\medskip
(2) Let $N$ be a submodule of $M$ containing $M_i$ for every $i\in I$. Then, $\sum_{i\in I}M_i$ is a submodule of $N$. 

\medskip
(b) Let $X$ be a scheme over $(\mathcal C,\otimes,1)$ and let $\mathcal M$ be a quasi-coherent sheaf on  $X$. Let $\{\mathcal M_i\}_{i\in I}$
be a family (finite or infinite) of quasi-coherent submodules of $\mathcal M$. Then, there exists a quasi-coherent submodule
$\sum_{i\in I}\mathcal M_i$ of $\mathcal M$  that is the sum of the quasi-coherent submodules $\{\mathcal M_i\}_{i\in I}$. 
\end{thm}

\begin{proof} (a) Since $A-Mod$ is an abelian category, it is well known that we can take sums of subobjects of any $M\in A-Mod$ (see, for instance,
\cite{PF}).  
Explicitly, the sum $\sum_{i\in I}M_i$ may be described as follows: we consider the induced morphism $q:\bigoplus_{i\in I}M_i\longrightarrow M$. Then, the sum $\sum_{i\in I}M_i$ is defined
to be the image of this morphism in the abelian category $A-Mod$; in other words, we set:
\begin{equation}\label{Eq2.10c}
\begin{array}{ll}
\sum_{i\in I}M_i& := Coker(Ker(q))=Coker(Ker(q:\bigoplus_{i\in I}M_i\longrightarrow M)\longrightarrow \bigoplus_{i\in I}M_i)\\
&\cong Ker(Coker(q))=Ker(M\longrightarrow Coker(q:\bigoplus_{i\in I}M_i\longrightarrow M))\\
\end{array}
\end{equation}  

\medskip
(b) For any $U\in ZarAff(X)$, we set $(\sum_{i\in I}\mathcal M_i)(U):=\sum_{i\in I}\mathcal M_i(U)$. We consider a Zariski open immersion $W \longrightarrow U$ of affines. Then, $\mathcal O_X(W)$ is a flat $\mathcal O_X(U)$-module. Since the sum in \eqref{Eq2.10c}
is defined in terms of colimits and finite limits, 
it now follows that $(\sum_{i\in I}\mathcal M_i)(W)=(\sum_{i\in I}\mathcal M_i)(U)\otimes_{\mathcal O_X(U)}\mathcal O_X(W)$. Hence, 
$\sum_{i\in I}\mathcal M_i$ is a quasi-coherent submodule of $\mathcal M$ that is the sum of the submodules $\mathcal M_i$, $i\in I$. 

\end{proof}

\medskip
Given an $A$-module $M$, we can consider the system of its finitely generated submodules ordered by inclusion. We will now show
that this system is filtered. 

\medskip
\begin{thm}\label{P3.2ct} (a) Let $A$ be a commutative monoid object of $(\mathcal C,\otimes,1)$ and let $M$ be an $A$-module. Then, the system of 
finitely generated submodules of $M$ ordered by inclusion is a filtered direct system. 

\medskip
(b) Let $X$ be a quasi-compact scheme over $(\mathcal C,\otimes,1)$ and let $\mathcal M$ be a quasi-coherent sheaf on $X$. Then, the system of finitely
generated quasi-coherent submodules of $\mathcal M$ is a filtered direct system. 

\end{thm}

\begin{proof} (a) Let $ M_1$, $ M_2$ be two finitely generated submodules of $M$. We will show that $M_1+M_2$ is also finitely generated. For this,
we choose a filtered direct system of monomorphisms $\{N_i\}_{i\in I}$ in $A-Mod$ and set $N=colim_{i\in I}N_i$. Since filtered colimits commute
with finite limits in $A-Mod$, it follows   that the canonical morphism $n_i:N_i\longrightarrow N$   is a monomorphism for each $i\in I$. We now choose a morphism $f:M_1+M_2\longrightarrow N$. 

\medskip
From \eqref{Eq2.10c}, we know that there exists an epimorphism $q':M_1\oplus M_2\longrightarrow M_1+M_2$. It is clear that $M_1\oplus M_2$ 
is finitely generated and hence the composition $f\circ q':M_1\oplus M_2\longrightarrow N$ factors through  $N_{i_0}$ for some $i_0\in I$. Further,
we see that the composition
\begin{equation}\label{3.3rvp}
Ker(q')\longrightarrow M_1\oplus M_2\longrightarrow N_{i_0}\overset{n_{i_0}}{\longrightarrow} N
\end{equation} is $0$. Then, since $n_{i_0}:N_{i_0}\longrightarrow N$ is a monomorphism, the composition $Ker(q')\longrightarrow M_1\oplus
M_2\longrightarrow N_{i_0}$ must be $0$. Finally, since $M_1+M_2=Coker(Ker(q')\longrightarrow M_1\oplus M_2)$, it follows that
the morphism $f:M_1+M_2\longrightarrow N$ factors through $N_{i_0}$. 

\medskip
(b) For finitely generated quasi-coherent submodules $\mathcal N$, $\mathcal P$ of $\mathcal M$, we consider the submodule 
$\mathcal N+\mathcal P$. From part (a) it follows that for each $U\in ZarAff(X)$, $(\mathcal N+\mathcal P)(U)
=\mathcal N(U)+\mathcal P(U)$ is a finitely generated submodule of $\mathcal M(U)$. This proves the result. 
\end{proof}

\medskip

\begin{lem}\label{L3.3xptA} (a) Let $f:A\longrightarrow B$ be a morphism in $Comm(\mathcal C)$. Let $M$ be a finitely generated
$A$-module. Then, $M\otimes_AB$ is finitely generated as a $B$-module. 

\medskip
(b) Suppose that every object of $\mathcal C$ can be expressed as a directed colimit of its finitely generated subobjects in $\mathcal C$. Then, 
for any $A\in Comm(\mathcal C)$, every $A$-module can be expressed as a directed colimit of its finitely generated $A$-submodules. 
\end{lem}

\begin{proof} (a) We consider a filtered system of monomorphisms $\{N_i\}_{i\in I}$ in $B-Mod$. Then, we have:
\begin{equation}
\begin{array}{l}
\underset{i\in I}{colim}\textrm{ }Hom_{B-Mod}(M\otimes_AB,N_i)\cong \underset{i\in I}{colim}\textrm{ }Hom_{A-Mod}(M,N_i)\\
\cong Hom_{A-Mod}(M,\underset{i\in I}{colim}\textrm{ }
N_i)\cong Hom_{B-Mod}(M\otimes_AB,\underset{i\in I}{colim}\textrm{ }N_i)\\
\end{array}
\end{equation} This proves the result. 

\medskip
(b) Let $M$ be an $A$-module. Then, we can express $M$ as a directed colimit $M\cong \underset{i\in I}{\varinjlim}\textrm{ }{M_i}$
of its finitely generated subobjects $\{M_i\}_{i\in I}$ in $\mathcal C$. Then, $\underset{i\in I}{\varinjlim}\textrm{ }{M_i\otimes A}\overset{\cong}
{\longrightarrow}M\otimes A$. Let $m:M\otimes A\longrightarrow M$ be the morphism that makes $M$ into an $A$-module. Now since the composition $M\otimes 1\longrightarrow M\otimes A\overset{m}{\longrightarrow} M$ is an isomorphism in $A-Mod$,
$m:M\otimes A\longrightarrow M$ is an epimorphism and hence so is the composition
$f:\underset{i\in I}{\varinjlim}\textrm{ }{M_i\otimes A}\overset{\cong}
{\longrightarrow}M\otimes A\overset{m}{\longrightarrow}M$. We now consider the canonical morphisms:
\begin{equation}
f_i:M_i\otimes A\longrightarrow \underset{i\in I}{\varinjlim}\textrm{ }{M_i\otimes A}\overset{\cong}
{\longrightarrow}M\otimes A\overset{m}{\longrightarrow}M
\end{equation} for each $i\in I$. For each $i\in I$, we also set $M_i'$ to be the $A$-submodule
$Im(f_i)$ of $M$. Since the colimit $f$ of the morphisms $\{f_i\}_{i\in I}$ is an epimorphism, 
it is clear that $M\cong \underset{i\in I}{\varinjlim}\textrm{ }M_i'$. From part (a), we know that
since $M_i$ is finitely generated in $\mathcal C$ (i.e, as a $1$-module), $M_i\otimes A$ is finitely generated
as an $A$-module. Further, since
each $M_i'$ is the image of the finitely generated $A$-module $M_i\otimes A$,
it follows as in the proof of Proposition \ref{P3.2ct}(a) that $M_i'$ is finitely generated in $A-Mod$. Thus, the result follows.

\end{proof}

\medskip
Throughout this section
and in the rest of this paper, we will make the following assumption on commutative monoid objects in $\mathcal C$:

\medskip
(C1) Any object of $\mathcal C$ can be expressed as a directed colimit of its finitely generated subobjects in $\mathcal C$. 
Equivalently, for any commutative monoid $A\in Comm(\mathcal C)$, any module $M\in A-Mod$ can be expressed as  the directed
colimit of its finitely generated submodules. 

\medskip
In other words, the condition (C1) says that the category $\mathcal C$ is ``locally finitely generated''. The theory of
locally finitely generated abelian categories and indeed the theory of locally finitely generated Grothendieck
categories is fairly well developed in the literature. For more on this, the reader may see, for example,
\cite{Gar}, \cite{Prest}, \cite{Prest3}, \cite{Prest2} or \cite{Sten}. 

\medskip

\begin{lem}\label{L3.3xpt} Let $X$ be a quasi-compact and semi-separated scheme over $(\mathcal C,\otimes,1)$ and let 
$\mathcal M$ be a quasi-coherent sheaf on $X$. Then, if $\{U_j\}_{j\in J}$ 
is an affine Zariski cover of $X$ such that $\mathcal M|_{U_j}$ is finitely generated over $U_j$ for each $j\in J$, $\mathcal M$ is a finitely
generated quasi-coherent sheaf on $X$.
\end{lem} 

\begin{proof}
Since $X$ is quasi-compact, we may suppose that $\{U_j=Spec(A_j)\}_{j\in J}$ is a finite affine cover. Then,
$\mathcal M(U_j)$ is finitely generated as an $A_j$-module for each $j\in J$. We choose
any $U=Spec(B)\in ZarAff(X)$. Further, since $X$ is semi-separated, each $U\times_XU_j=Spec(B_j)$ is affine
and the $Spec(B_j)$, $j\in J$ form a finite affine cover of $Spec(B)$. From Lemma \ref{L3.3xptA}(a), it follows that $\mathcal M_j:=\mathcal M(U\times_XU_j)=
\mathcal M(U_j)\otimes_{A_j}B_j$ is finitely generated as a $B_j$-module. 

\medskip
We now consider a filtered system of monomorphisms $\{N_i\}_{i\in I}$ in $B-Mod$ and a morphism
$g:\mathcal M(U)\longrightarrow colim_{i\in I}\textrm{ }N_i$ of $B$-modules. Since each $B_j$ is flat as a $B$-module,
$\{N_i\otimes_BB_j\}_{i\in I}$ is a filtered system of monomorphisms in $B_j-Mod$ for each $j\in J$.  Since $\mathcal M_j$ is a finitely
generated $B_j$-module, $J$ is finite and $I$ is filtered, we can choose $i_0\in I$ such that
$g\otimes_BB_j:\mathcal M_j=\mathcal M(U\times_XU_j)=\mathcal M(U)\otimes_BB_j\longrightarrow colim_{i\in I}N_i\otimes_BB_j$ 
factors through $N_{i_0}\otimes_BB_j$ for each $j\in J$. Further, since $N_{i_0}\otimes_BB_j\longrightarrow 
colim_{i\in I}N_i\otimes_BB_j$  is a monomorphism, the morphism $g\otimes_BB_j:\mathcal M_j=\mathcal M(U\times_XU_j)=\mathcal M(U)\otimes_BB_j\longrightarrow colim_{i\in I}N_i\otimes_BB_j$ 
factors uniquely through $N_{i_0}\otimes_BB_j$. Finally, since $\{Spec(B_j)\longrightarrow Spec(B)=U\}_{j\in J}$ forms a finite affine cover,
it follows from \cite[Corollaire 2.11]{TV} that:
\begin{equation}
\begin{array}{c}
\mathcal M(U)=lim\left(\prod_{j\in J}\mathcal M(U)\otimes_{B}B_j \overset{\longrightarrow}{\underset{\longrightarrow}{ }}\prod_{j,j'\in J}
\mathcal M(U)\otimes_BB_j\otimes_BB_{j'}\right) \\
N_{i_0}=lim\left(\prod_{j\in J}N_{i_0}\otimes_{B}B_j \overset{\longrightarrow}{\underset{\longrightarrow}{ }}\prod_{j,j'\in J}
N_{i_0}\otimes_BB_j\otimes_BB_{j'}\right) \\ 
\end{array}
\end{equation} and hence the morphisms $\mathcal M(U)\otimes_BB_j\longrightarrow N_{i_0}\otimes_BB_j$ can be glued
together to give a morphism $\mathcal M(U)\longrightarrow N_{i_0}$. 
 
\end{proof}

\medskip
\begin{lem}\label{L3.03} Let $U$ be a scheme over $(\mathcal C,\otimes,1)$ that is also quasi-compact. Let $\mathcal N$ be a finitely
generated quasi-coherent sheaf on $U$. Suppose that $\mathcal N$ can be expressed as a filtered colimit $\mathcal N
=\underset{\lambda\in \Lambda}{\varinjlim}\textrm{ }\mathcal N_\lambda$ of quasi-coherent submodules $\mathcal N_\lambda$. 
Then, there exists $\lambda_0\in \Lambda$ such that $\mathcal N=\mathcal N_{\lambda_0}$. 
\end{lem}

\begin{proof} Since $U$ is quasi-compact, we can choose a finite affine cover $U_i=Spec(A_i)$, $1\leq i\leq m$ of $U$. Then, for each
$1\leq i\leq m$, we can express $\mathcal N(U_i)$ as a filtered colimit $\mathcal N(U_i)\cong \underset{\lambda\in\Lambda}{\varinjlim}
\textrm{ }\mathcal N_\lambda(U_i)$ of submodules. Since 
$\mathcal N$ is finitely generated,  there exists $\lambda_i\in \Lambda$ such that the identity map $\mathcal N(U)\longrightarrow \mathcal N(U)$ 
 factors through the monomorphism $\mathcal N_{\lambda_i}(U_i)\longrightarrow \mathcal N(U_i)$. 
Hence, $\mathcal N_{\lambda_i}(U_i)\longrightarrow \mathcal N(U_i)$ is also an epimorphism. It follows that $\mathcal N_{\lambda_i}(U_i)
\cong \mathcal N(U_i)$ (since it is both an epimorphism and a monomorphism and $A_i-Mod$ is an abelian category). 
Since $U_i$ is affine, we now know that $\mathcal N_{\lambda_i}|U_i=\mathcal N|U_i$. 

\medskip
Finally, since $\Lambda$ is filtered and we have only finitely many $U_i$, it follows that we can choose $\lambda_0\in \Lambda$ such that $\mathcal N_{\lambda_0}|U_i=\mathcal N|U_i$ for
all $1\leq i\leq m$. Since the $U_i$ form an affine cover of $U$, it now follows that $\mathcal N_{\lambda_0}=\mathcal N$. 

\end{proof}

\medskip
\begin{thm}\label{P3.3} Let $A\in Comm(\mathcal C)$ be a commutative monoid object and let $X=Spec(A)$ be the affine
scheme corresponding to $A$. Let $U$ be a quasi-compact scheme over $(\mathcal C,\otimes,1)$ along with a Zariski open immersion $i:U\longrightarrow X$  that is also quasi-compact. 
Then, for any quasi-coherent sheaf $\mathcal M$ on $X$ and a finitely generated quasi-coherent submodule $\mathcal N$ of
the restriction $i^*\mathcal M=\mathcal M|_U$, there exists a finitely generated quasi-coherent submodule $\mathcal N'$ of
$\mathcal M$ such that $\mathcal N'|_U=\mathcal N$. 
\end{thm}

\begin{proof} Since $X$ is affine and hence semi-separated, it is clear that the open immersion $i$ is semi-separated. Now, since 
$i_*$ is a right adjoint, it follows that $i_*\mathcal N$ is a quasi-coherent submodule of $i_*i^*\mathcal M=i_*(\mathcal M|_U)$. 
We now set:
\begin{equation}\label{3.5}
\overline{\mathcal N}(V):=lim(\mathcal M(V)\longrightarrow (i_*i^*\mathcal M)(V)\longleftarrow i_*\mathcal N(V))\qquad
\forall\textrm{ }V\in ZarAff(X)
\end{equation} Since $\overline{\mathcal N}$ is defined in terms of a finite limit in \eqref{3.5}, it follows that $\overline{\mathcal N}
\in QCoh(X)$. Further, since $i_*\mathcal N\longrightarrow i_*i^*\mathcal M$ is a monomorphism, it follows from the limit
in \eqref{3.5} that $\overline{\mathcal N}(V)\longrightarrow \mathcal M(V)$ is a monomorphism, i.e., 
$\overline{\mathcal N}$ is a quasi-coherent submodule of $\mathcal M$.  We now consider the restriction
$\overline{\mathcal N}|_U$. For any $W\in ZarAff(U)$, it follows from \eqref{3.5} that:
\begin{equation}\label{3.6ct}
(\overline{\mathcal N}|_U)(W)=\overline{\mathcal N}(W)=lim(\mathcal M(W)\longrightarrow (i_*i^*\mathcal M)(W)\longleftarrow i_*\mathcal N(W))
\end{equation} Since $W\in ZarAff(U)$, we notice that  $\mathcal M(W)=i^*\mathcal M(W)$, $(i_*i^*\mathcal M)(W)=((i^*i_*)(i^*\mathcal M))(W)$and $ i_*\mathcal N(W)=(i^*i_*\mathcal N)(W)$. 
Then, it follows from Proposition \ref{P2.5qp} that
$(i_*i^*\mathcal M)(W)=i^*\mathcal M(W)$ and $i_*\mathcal N(W)=\mathcal N(W)$. Combining this with \eqref{3.6ct}, it follows that
\begin{equation}\label{3.7ct}
(\overline{\mathcal N}|_U)(W)=lim(i^*\mathcal M(W)\longrightarrow i^*\mathcal M(W)\longleftarrow \mathcal N(W))
=\mathcal N(W)\qquad\forall\textrm{ }W\in ZarAff(U)
\end{equation} Thus, $\overline{\mathcal N}$ is a quasi-coherent submodule of $\mathcal M$ such that $\overline{\mathcal N}|_U=\mathcal N$. Further,
since $X=Spec(A)$ is affine, we know that $\overline{\mathcal N}\in QCoh(X)$ corresponds to an $A$-module $\overline{N}$. Using Proposition \ref{P3.2ct}(a) and  condition (C1) towards the beginning of this section, it follows
that the $A$-module $\overline{N}$ may be expressed as a filtered colimit of its finitely generated submodules.   Accordingly, we can express $\overline{\mathcal N}$ as a filtered colimt $\overline{\mathcal N}=\underset{\lambda\in \Lambda}{\varinjlim}\textrm{ }
\mathcal N_\lambda$, where $\mathcal N_\lambda$ is the quasi-coherent submodule of $\overline{\mathcal N}$ corresponding to a finitely
generated submodule $N_\lambda$ of $\overline{N}$. From Lemma \ref{L3.3xptA}(a), it follows that the quasi-coherent module
$\mathcal N_\lambda$ on $Spec(A)$ corresponding to a finitely generated submodule module $N_\lambda$ of
$\overline{N}$ is also finitely generated as a quasi-coherent sheaf on $X=Spec(A)$. Since the restriction $i^*$ is a left adjoint, it now follows that:
\begin{equation}\label{3.8tv}
\mathcal N=\overline{\mathcal N}|_U=\underset{\lambda\in \Lambda}{\varinjlim}\textrm{ }\mathcal N_\lambda|_U
\end{equation} As noted before, the system $\Lambda$ is filtered. Since $\mathcal N$ is finitely generated, it now follows
from Lemma \ref{L3.03} that  there exists some $\lambda_0\in \Lambda$ such that $\mathcal N=\mathcal N_{\lambda_0}|U$. This proves
the result. 

\end{proof}

\medskip

Let $X$ be a quasi-compact scheme over $(\mathcal C,\otimes,1)$ and let $\{U_i\}_{1\leq i\leq n}$ be a finite affine cover of $X$. Then, we know that the scheme $X$ can actually be written as a quotient $Y/R$, where
$Y=\coprod_{i=1}^nU_i$ is the disjoint union of the affine schemes $U_i$ and $R\subseteq Y\times Y$ is an equivalence relation
in $Sh(Aff_{\mathcal C})$ satisfying certain conditions described in \cite[Proposition 2.18]{TV}. We now recall the following construction from
\cite[$\S$ 4]{AB2}: For any $1\leq m\leq n$, we set $Y_m:=\coprod_{i=1}^mU_i$ and consider the equivalence relation $R_m$ on
$Y_m$ defined as follows:
\begin{equation}\label{QE3.9}
\begin{CD}
R_m:=R\times_{(Y\times Y)}(Y_m\times Y_m) @>>> Y_m\times Y_m\\
@VVV @VVV \\
R @>>> Y\times Y \\
\end{CD}
\end{equation} Then, we have shown in \cite[Lemme 4.3]{AB2} the following result:

\medskip
\begin{thm}\label{P3.5qm} For any given $1\leq m\leq n$, consider the disjoint union $Y_m=\coprod_{i=1}^mU_i$ as well as the equivalence relation $R_m$ as defined in \eqref{QE3.9} above. Then, the induced
morphism $p_m:X_m\longrightarrow X$ from the scheme $X_m:=Y_m/R_m$ is a Zariski open immersion. 
\end{thm}

\medskip
We note that the result of Proposition \ref{P3.3} applies in particular to open immersions $U=Spec(B)\longrightarrow X=Spec(A)$ of affine schemes. We will now extend this to the case where $X$ is any    quasi-compact and semi-separated scheme over $(\mathcal C,\otimes,1)$ and $U\in ZarAff(X)$. 

\medskip
\begin{thm}\label{P3.6cp} Let $X$ be a quasi-compact and semi-separated scheme over $(\mathcal C,\otimes,1)$ and let $U\in ZarAff(X)$. Then, for any quasi-coherent sheaf $\mathcal M$ on $X$ and any finitely generated quasi-coherent submodule $\mathcal N$ of
the restriction $\mathcal M|_U$, there exists a finitely generated quasi-coherent submodule $\mathcal N'$ of
$\mathcal M$ such that $\mathcal N'|_U=\mathcal N$. 
\end{thm}

\begin{proof} Since $X$ is quasi-compact, we can choose a finite affine cover $\{U_i\}$, $1\leq i\leq n$ of $X$ with $U_1=U$. For any $1\leq m\leq n$,
we consider the schemes $X_m$ and the open immersions $p_m:X_m\longrightarrow X$ as described above. We set $\mathcal N_1=\mathcal N$. For every such $m$, we want to define
a finitely generated quasi-coherent submodule $\mathcal N_m$ of $\mathcal M|_{X_m}$ such that $\mathcal N_m|_{X_i}=\mathcal N_i$ for all
$1\leq i\leq m$.  Suppose that we have successfully defined such sheaves $\mathcal N_i$ 
for every $i< m$. We will now describe how they can be used to define $\mathcal N_{m}$ on $X_{m}$. 

\medskip
We now choose some $W\in ZarAff(U_m)$ and consider the following fiber squares:
\begin{equation}\label{Qe3.10}
\begin{CD}
U_i\times_XW@>>> (X_{m-1}\times_XU_m)\times_XW@>>> W \\
@VVV @VVV @VVV \\
U_i\times_XU_m@>>> X_{m-1}\times_XU_m@>>> U_m\\
@VVV @VVV @VVV \\
U_i@>>> X_{m-1} @>p_{m-1}>> X \\
\end{CD} \qquad \forall\textrm{ }1\leq i<m
\end{equation} From Proposition \ref{P3.5qm}, we know that 
$p_{m-1}:X_{m-1}\longrightarrow X$ is a Zariski open immersion and hence so is $(X_{m-1}\times_XU_m)\longrightarrow U_m$. Since $X$ is semi-separated, $U_i\times_XW$ is affine for each $1\leq i<m$. From \eqref{Qe3.10}, it follows
that for any $W\in ZarAff(U_m)$, the fiber product $(X_{m-1}\times_XU_m)\times_XW$ has a finite affine 
covering $\{U_i\times_XW\}_{1\leq i<m}$ and is therefore quasi-compact. It follows that the Zariski open immersion
$(X_{m-1}\times_XU_m)\longrightarrow U_m$ is quasi-compact.  In particular, since $U_m$ is affine, i.e., 
$U_m\in ZarAff(U_m)$, it also follows that
$(X_{m-1}\times_XU_m)$ is quasi-compact.  We can now apply Proposition \ref{P3.3} to the   immersion
$(X_{m-1}\times_XU_m)\longrightarrow U_m$. 

\medskip
Accordingly, it follows that there exists a finitely generated quasi-coherent submodule $\mathcal N_m'$ of $\mathcal M|_{U_m}$ such that:
\begin{equation}\label{Qe3.11}
\mathcal N_m'|_{(X_{m-1}\times_XU_m)}=\mathcal N_{m-1}|_{(X_{m-1}\times_XU_m)}
\end{equation} From \eqref{Qe3.11}, it follows that the quasi-coherent sheaves $\mathcal N_{m-1}$  and $\mathcal N_m'$ on $X_{m-1}$ and $U_m$ respectively
agree on $X_{m-1}\times_XU_m$. Hence, we can consider the  quasi-coherent submodule $\mathcal N_m$ of  $\mathcal M|_{X_m}$ such that 
$\mathcal N_m|_{X_{m-1}}=\mathcal N_{m-1}$ and $\mathcal N_m|_{U_m}=\mathcal N_m'$. Further, the restriction of $\mathcal N_m$ to
$X_{m-1}$ and $U_m$ being finitely generated, it follows from Lemma \ref{L3.3xpt} that $\mathcal N_m$ is a finitely generated
quasi-coherent sheaf on $X_m$. 

\medskip It follows that $\mathcal N':=\mathcal N_n$ is a finitely generated
quasi-coherent submodule of $\mathcal M$ such that $\mathcal N'|_{U}=\mathcal N$. 

\end{proof}

\medskip
We are now ready to show that any quasi-coherent sheaf on a quasi-compact and semi-separated scheme is a filtered direct limit of
its finitely generated quasi-coherent submodules.

\medskip
\begin{Thm}\label{Thm3.8} Let $X$ be a quasi-compact and semi-separated scheme over $(\mathcal C,\otimes,1)$ and let 
$\mathcal M$ be a quasi-coherent sheaf on $X$. Then, $\mathcal M$ can be expressed as a filtered direct limit
of its finitely generated quasi-coherent submodules. 
\end{Thm}

\begin{proof} We consider the  system $\{\mathcal M_\lambda\}_{\lambda\in\Lambda_M}$ of finitely generated quasi-coherent submodules of $\mathcal M$. From Proposition \ref{P3.2ct}(b), we know
that $\Lambda_M$ is filtered. We have a natural morphism:
\begin{equation}
f:\mathcal M':=\underset{\lambda\in \Lambda_M}{\varinjlim}\textrm{ }\mathcal M_{\lambda}
\longrightarrow \mathcal M\qquad \mathcal M'(U):=\underset{\lambda\in \Lambda_M}{\varinjlim}\textrm{ }\mathcal M_{\lambda}(U)\qquad\forall \textrm{ }U\in ZarAff(X)
\end{equation} Since each $\mathcal M_\lambda(U)\longrightarrow \mathcal M(U)$ is a monomorphism,   $f(U):\mathcal M'(U)\longrightarrow 
\mathcal M(U)$ is also a monomorphism (because filtered colimits commute with finite limits in $\mathcal O_X(U)-Mod$). Hence, $f:\mathcal M'\longrightarrow 
\mathcal M$ is a monomorphism in $QCoh(X)$. 

\medskip We now consider the restriction $\mathcal M|_U$ of $\mathcal M$ to some $U\in ZarAff(X)$.  Since $U$ is affine, we can express $\mathcal M|_U$ as a filtered
colimit of its finitely generated submodules $\mathcal M'_\lambda$, $\lambda\in \Lambda_U$. From Proposition \ref{P3.6cp}, it follows that
we can choose finitely generated quasi-coherent submodules $\mathcal M''_{\lambda}$, $\lambda\in \Lambda_U$  of $\mathcal M$ such that
$\mathcal M''_\lambda |_U=\mathcal M'_\lambda$.  We let $\Lambda'_M$ be the filtered subsystem of $\Lambda_M$ consisting
of all finite sums of the finitely generated submodules $\{\mathcal M''_{\lambda}\}_{\lambda\in \Lambda_U, U\in ZarAff(X)}$. Then, it is clear
that the natural morphism:
\begin{equation}
g:\mathcal M'':=\underset{\lambda\in \Lambda_M'}{\varinjlim}\textrm{ }\mathcal M''_{\lambda}
\longrightarrow \mathcal M
\end{equation} is an isomorphism. Further, the isomorphism $g:\mathcal M''\overset{\cong}{\longrightarrow} \mathcal M$ factors through $f:
\mathcal M'\longrightarrow \mathcal M$ from which it follows that $f$ is also an epimorphism in $QCoh(X)$. Since
$QCoh(X)$ is an abelian category (see \cite[Proposition 2.9]{AB2}), it follows that $f$ is an isomorphism. 

\end{proof} 

\medskip

\medskip
\section{Noetherian commutative monoids and field objects over $(\mathcal C,\otimes,1)$} 

\medskip

\medskip

In this section, we will  introduce and study Noetherian schemes and field objects over $(\mathcal C,\otimes,1)$. We will define
Noetherian commutative monoids to be those for which every finitely generated module is also ``finitely presented'' (see Definition 
\ref{Def4.1}). Our notion of a ``field object'' in $(\mathcal C,\otimes,1)$ is presented in Definition \ref{D4.2}. We will see that with the notion
of field object as in Definition \ref{D4.2},  we can recover several of the usual properties of a field.  Thereafter, in Section 5,
we will show how the points of a Noetherian, quasi-compact and semi-separated scheme $X$ over a field object $K$ 
can be recovered from certain kinds of cocontinuous symmetric monoidal functors between categories of quasi-coherent sheaves.

\medskip
\begin{defn}\label{Def4.1} Let $A\in Comm(\mathcal C)$ be a commutative monoid object of $(\mathcal C,\otimes,1)$. For any
$n\geq 1$, we let $A^n$ denote the direct sum of $n$-copies of $A$. 
 We will say that $A$ is Noetherian
if for any finitely generated $A$-module $M$, there exists a morphism $q:A^m\longrightarrow A^n$ for some $m$, $n\geq 1$ such that $M$  can be  expressed as a colimit: 
\begin{equation}\label{4.1}
colim(0\longleftarrow A^m \overset{q}{\longrightarrow}A^n)\overset{\cong}{\longrightarrow}M
\end{equation} 

\medskip
We will say that a scheme $X$ over $(\mathcal C,\otimes,1)$ is Noetherian if for any affine $U=Spec(A)\in ZarAff(X)$, $A$ is a Noetherian commutative monoid object 
in the above sense. 
\end{defn}

\medskip 
We note here that in previous work in \cite{ABp}, we have explored
other notions of ``Noetherian'' for monoids objects and schemes over
symmetric monoidal categories. 
Before we proceed further, we must show that if $A\in Comm(\mathcal C)$ is Noetherian, the corresponding affine scheme 
$Spec(A)$ is a Noetherian scheme in the sense of Definition \ref{Def4.1}. In order to prove this, we mention here the following Lemma that
is well known in the case of ordinary commutative rings (see, for example, \cite[Chapitre IV]{Lazard}). 

\medskip
\begin{lem}\label{Read1} Let $f:A\longrightarrow B$ be an epimorphism of commutative monoid objects in $(\mathcal C,\otimes,1)$. 
Then, we have an isomorphism $B\cong B\otimes_AB$. 
\end{lem} 

\begin{proof} Let $f_1$, $f_2:B\longrightarrow C$ be morphisms in $Comm(\mathcal C)$ satisfying $f_1\circ f=f_2\circ f$. Then, $f$ being
an epimorphism in $Comm(\mathcal C)$, we must have $f_1=f_2$. In other words, $B$ is the pushout in $Comm(\mathcal C)$ of the diagram
$B\overset{f}{\longleftarrow} A \overset{f}{\longrightarrow}B$. On the other hand, we know that the pushout of 
$B\overset{f}{\longleftarrow} A \overset{f}{\longrightarrow}B$ is $B\otimes_AB$. It follows that $B\cong B\otimes_AB$. 

\end{proof}

\medskip
\begin{thm}\label{P4.2xr} Let $f:A\longrightarrow B$ be an epimorphism of commutative monoids in $(\mathcal C,\otimes,1)$ such that
$B$ is a flat $A$-module.  Then, 
if $A\in Comm(\mathcal C)$ is Noetherian, so is $B$. In particular, if $A\in Comm(\mathcal C)$ is a Noetherian commutative monoid object,
$Spec(A)$ is a Noetherian scheme. 
\end{thm}

\begin{proof} Let $A$ be Noetherian and let $N$ be a finitely generated $B$-module. Then, $N$ is an $A$-module by ``restriction of scalars''. 
From condition (C1) in Section 3, we know that as an $A$-module, $N$ may be expressed
as the filtered colimit of its finitely generated $A$-submodules $\{N_i\}_{i\in I}$. It follows that:
\begin{equation}\label{4.2dc9}
N\otimes_AB\cong \underset{i\in I}{\varinjlim}\textrm{ }N_i\otimes_AB
\end{equation} On the other hand, since $f:A\longrightarrow B$ is an epimorphism in $Comm(\mathcal C)$, 
we know from Lemma \ref{Read1} that $B\otimes_AB\cong B$. Combining this with \eqref{4.2dc9}, 
we have:
\begin{equation}\label{4.3dc9}
N\cong N\otimes_BB\cong N\otimes_B(B\otimes_AB)\cong N\otimes_AB\cong \underset{i\in I}{\varinjlim}\textrm{ }N_i\otimes_AB
\end{equation} Since $B$ is a flat $A$-module, $\{N_i\otimes_AB\}_{i\in I}$ is a filtered system of monomorphisms in
$B-Mod$. Then, $N$ being a finitely generated $B$-module, there exists $i_0\in I$ such that
$N\cong N_{i_0}\otimes_AB$. 
 
 \medskip
 Finally, since $A$ is Noetherian and $N_{i_0}$ is finitely generated as an $A$-module, there exists a 
 morphism $q:A^m\longrightarrow A^n$ for some $m,n\geq 1$ such that $N_{i_0}$ can be expressed as a colimit: 
\begin{equation}\label{4.4b747}
colim(0\longleftarrow A^m \overset{q}{\longrightarrow}A^n)\overset{\cong}{\longrightarrow}N_{i_0}
\end{equation} From \eqref{4.4b747}, it follows that:
\begin{equation}
N\cong N_{i_0}\otimes_AB\cong colim(0\longleftarrow A^m \overset{q}{\longrightarrow}A^n)\otimes_AB\cong colim\left(
\begin{CD}0 @<<< B^m @>q\otimes_AB>>B^n\end{CD}\right)
\end{equation} Hence, $B$ is Noetherian. 

\medskip
In particular, if $f:A\longrightarrow B$ induces a Zariski open immersion of affine schemes, we know that
$f:A\longrightarrow B$ must be an epimorphism in $Comm(\mathcal C)$ (see \cite[D\'{e}finition 2.9]{TV}) and $B$ must be a flat 
$A$-module. Hence, $Spec(A)$ is Noetherian. 

\end{proof}

\medskip
\begin{defn}\label{D4.2} Let $K\in Comm(\mathcal C)$ be a commutative monoid object of $(\mathcal C,\otimes,1)$ such that $K\ne 0$. We will say that $K$ is a field object of $(\mathcal C,\otimes,1)$ if 
$K$ is Noetherian and  any monomorphism $I\longrightarrow K$ in $K-Mod$ is either an isomorphism or zero. 
\end{defn}

\medskip
Since $K$ is a commutative monoid object, it follows from the well known ``Eckmann-Hilton argument'' (see \cite{EH}) that $\mathcal E(K):=Hom_{K-Mod}(K,K)$
is a commutative ring. In particular, since the field object $K$ has no subobjects in $K-Mod$ (other than $0$ and $K$), 
every non-zero morphism in $Hom_{K-Mod}(K,K)$ is an isomorphism. It follows that for a field object $K$ in $(\mathcal C,\otimes,1)$, $\mathcal E(K)
$ is a field. Further, for any $K$-module $M$, $Hom_{K-Mod}(K,M)$ becomes a vector space
over $\mathcal E(K)$. In the rest of this section, we will further study field objects in $(\mathcal C,\otimes,1)$.

\medskip
\begin{lem}\label{Lem5.4lg} Let $K$ be a field object in $(\mathcal C,\otimes,1)$ and let $f:L\longrightarrow K$ be an epimorphism
in $K-Mod$. Then, the induced map $Hom_{K-Mod}(K,f):Hom_{K-Mod}(K,L)\longrightarrow Hom_{K-Mod}(K,K)$ is a surjection. 
\end{lem} 

\begin{proof} For any $M\in K-Mod$, let $f_M:Hom_{K-Mod}(M,L)\longrightarrow Hom_{K-Mod}(M,K)$ denote the morphism induced by $f$. 
Now, since $\mathcal E(K)=Hom_{K-Mod}(K,K)$ is a field, if  we assume that $f_K:Hom_{K-Mod}(K,L)\longrightarrow Hom_{K-Mod}(K,K)$ 
is not a surjection of vector spaces, it must be $0$. Then, for any $n\geq 1$, the morphism
$f_{K^n}:Hom_{K-Mod}(K^n,L)\longrightarrow Hom_{K-Mod}(K^n,K)$ induced by $f$ is $0$. Further, since $K$ is Noetherian, any finitely
generated $K$-module $M$ can be expressed as a colimit of the form:
$
M\cong colim(0\longleftarrow K^m \overset{q}{\longrightarrow}K^n)
$. We now have the following commutative diagram:
\begin{equation}
\begin{CD}
lim(0\longrightarrow Hom_{K-Mod}(K^m,L)\longleftarrow Hom_{K-Mod}(K^n,L)) @>\cong>> Hom_{K-Mod}(M,L)\\
@V0VV @Vf_MVV \\
lim(0\longrightarrow Hom_{K-Mod}(K^m,K)\longleftarrow Hom_{K-Mod}(K^n,K)) @>\cong>> Hom_{K-Mod}(M,K)\\
\end{CD}
\end{equation} We see that $f_M=0:Hom_{K-Mod}(M,L)\longrightarrow Hom_{K-Mod}(M,K)$ for any finitely generated
$K$-module $M$. Finally, since every $K$-module can be expressed as a filtered colimit of its finitely generated submodules,
it follows that $f_M=0$ for each $M\in K-Mod$.  From Yoneda Lemma, it now follows that $f=0:L\longrightarrow K$. Since
$K\ne 0$, this contradicts the fact that $f$ is an epimorphism. 
\end{proof}

\medskip
\begin{thm} \label{P4.6vad} Let $K$ be a field object in $(\mathcal C,\otimes,1)$. Then:

\medskip
(a) For any $n\geq 1$, $K^n$ is a projective object of $K-Mod$.

\medskip
(b) Let $f:M\longrightarrow N$ be a morphism of finitely generated $K$-modules. Choose any morphisms $q_M:K^{m'}\longrightarrow K^m$
and $q_N:K^{n'}\longrightarrow K^n$ such that $M$ and $N$ may be expressed as 
\begin{equation}
M\cong colim(0\longleftarrow K^{m'} \overset{q_M}{\longrightarrow}K^m)
\qquad N\cong colim(0\longleftarrow K^{n'} \overset{q_N}{\longrightarrow}K^n)
\end{equation}
 respectively. Then, we have morphisms $g:K^m\longrightarrow K^n$ and
 $h:K^{m'}\longrightarrow K^{n'}$ such that the following diagram is commutative:
 \begin{equation}\label{4.8bur}
 \begin{CD}
 K^{m'} @>q_M>> K^m @>>> M @>>> 0\\
 @VhVV @VgVV @VfVV @. \\
 K^{n'} @>q_N>> K^n @>>> N @>>> 0\\
 \end{CD}
 \end{equation}
\end{thm}

\begin{proof} (a) It suffices to show that $K$ is projective. Let $p:S\longrightarrow T$ be an epimorphism in $K-Mod$ and take any
$f\in Hom_{K-Mod}(K,T)$. Set $T':=Im(f)$. We now define a pullback square:
\begin{equation}\label{4.9but}
\begin{CD}
S'@>p'>> T'=Im(f) \\
@VVV @VVV \\
S @>p>> T \\
\end{CD}
\end{equation} Since $K-Mod$ is an abelian category and $p:S\longrightarrow T$ is an epimorphism, its 
pullback $p':S'\longrightarrow T'$ is also an epimorphism (see, for example, \cite[$\S$ 2.54]{PF}). Now, if $T'=0$, the morphism $f=0:K\longrightarrow T$ lifts to $S$. Otherwise, since $K$ has no   subobjects other
than $0$ and $K$, we must have $T'\cong K$. It now follows from Lemma \ref{Lem5.4lg} that the morphism $Hom_{K-Mod}(K,S')
\longrightarrow Hom_{K-Mod}(K,T')\cong Hom_{K-Mod}(K,K)$ induced by $p':S'\longrightarrow T'\cong K$ must be an epimorphism. Therefore,
the morphism $f:K\longrightarrow T'=Im(f)$ can be lifted to a morphism from $K$ to $S'$. Composing with the inclusion $S'\hookrightarrow
S$, we see that the morphism $f:K\longrightarrow Im(f)=T'\hookrightarrow T$ can be lifted to a morphism from $K$ to $S$. Hence $K$
is projective. 

\medskip
(b) The horizontal rows of the diagram \eqref{4.8bur} are exact  in $K-Mod$. Since $K$ is projective, the proof of part (b)
follows exactly as in the usual case of finitely presented modules over a commutative ring. 

\end{proof}

\medskip
\begin{lem}\label{Lem4.4qg} Every monomorphism (resp. epimorphism) in $K-Mod$ is a split monomorphism (resp. a split epimorphism). 
\end{lem}

\begin{proof} Let $i:M\longrightarrow N$ be a monomorphism in $K-Mod$. Then,   $i_K:=Hom_{K-Mod}(K,i):Hom_{K-Mod}(K,M)\longrightarrow Hom_{K-Mod}(K,N)$
is a monomorphism of vector spaces. Hence, there exists a morphism $p_K:Hom_{K-Mod}(K,N)\longrightarrow Hom_{K-Mod}(K,M)$ such that
$p_K\circ i_K=1$. Since $K$ is Noetherian, any finitely generated module $F$ in $K-Mod$ can be expressed as a colimit
$F\cong colim(0\longleftarrow K^m\longrightarrow K^n)$. Then, since 
\begin{equation}
\begin{array}{c}
lim(0\longrightarrow Hom_{K-Mod}(K^m,M)\longleftarrow Hom_{K-Mod}(K^n,M))  \cong  Hom_{K-Mod}(F,M)\\
lim(0\longrightarrow Hom_{K-Mod}(K^m,N)\longleftarrow Hom_{K-Mod}(K^n,N))  \cong  Hom_{K-Mod}(F,N)\\
\end{array}
\end{equation} the morphisms $i_K$ and $p_K$ induce:
\begin{equation}\label{4.6vcr}
\begin{CD}
Hom_{K-Mod}(F,M)@>i_F>> Hom_{K-Mod}(F,N)@>p_F>> Hom_{K-Mod}(F,M) 
\end{CD} \qquad p_F\circ i_F=1
\end{equation} From Proposition \ref{P4.6vad}(b), we see that the morphism
$p_F$ does not depend on the choice of the presentation $F\cong  colim(0\longleftarrow K^m\longrightarrow K^n)$.   Further, since any module  in $K-Mod$ can be expressed as a colimit of its finitely generated submodules, 
we have morphisms $p_F$ and $i_F$ as in \eqref{4.6vcr} for any $F\in K-Mod$. From Yoneda lemma, it follows that the morphisms
$p_F$ are induced by 
 a morphim $p:N\longrightarrow M$ with $p\circ i=1$. Hence, every monomorphism in $K-Mod$ splits. 
Further, since every epimorphism in $K-Mod$ can be fitted into a short exact sequence, it follows that
every epimorphism in $K-Mod$ also splits. 
\end{proof}

\medskip
\begin{thm}\label{P4.5xpb} Let $K$ be a field object in $(\mathcal C,\otimes,1)$. Then, every finitely generated $K$-module is isomorphic
to  a direct sum $K^m$ for some $m\geq 1$. 
\end{thm}

\begin{proof} Since $K$ is Noetherian, for every finitely generated module $M$, there exists an epimorphism 
$K^n\longrightarrow M$ for some $n\geq 1$ as in \eqref{4.1}. By Lemma \ref{Lem4.4qg}, the epimorphism
splits and hence there is an inclusion $M\hookrightarrow K^n$. Accordingly,  $Hom_{K-Mod}(K,M)$
is a subspace of the finite dimensional vector space $Hom_{K-Mod}(K,K^n)=\mathcal E(K)^n$. Hence, we have an isomorphism
$i_K:Hom_{K-Mod}(K,M)\overset{\cong}{\longrightarrow} \mathcal E(K)^m=Hom_{K-Mod}(K,K^m)$ for some $m\leq n$. 

\medskip
Then, by expressing any finitely generated $K$-module $F$ as a colimit $F\cong colim(0\longleftarrow K^p\longrightarrow K^q)$ ,  we see that  $i_K$ induces a morphism $i_F:Hom_{K-Mod}(F,M) {\longrightarrow}  Hom_{K-Mod}(F,K^m)$
for each finitely generated $F$ in $K-Mod$. Again, from Proposition \ref{P4.6vad}(b), it follows that $i_F$ is independent of the choice
of the presentation $F\cong colim(0\longleftarrow K^p\longrightarrow K^q)$. Further, since any module  in $K-Mod$ can be expressed as a colimit of its finitely generated submodules, 
we have a morphism $i_F: Hom_{K-Mod}(F,M) {\longrightarrow}  Hom_{K-Mod}(F,K^m)$  for any $F\in K-Mod$.  Since $i_K$ is an isomorphism, so is each $i_F$. From Yoneda lemma, it now follows that the isomorphisms $i_F$ are induced by 
  an isomorphism $i:M\overset{\cong}{\longrightarrow} K^m$. 
\end{proof}

\medskip
We also record here the following result that will be useful to us in Section 5.

\medskip
\begin{thm} \label{P4.6dj}  Let $K$ be a field object in $(\mathcal C,\otimes,1)$. Then, if $i:U\longrightarrow Spec(K)$ is a Zariski
open immersion, then either $U=Spec(0)$ or $U=Spec(K)$ and $i$ is an isomorphism. 
\end{thm}

\begin{proof} First, we suppose that $U$ is affine. Let $U=Spec(A)$ and suppose that $A\ne 0$. Since $K$ has no proper subobjects, the induced morphism $K\longrightarrow A$ must be a monomorphism in $K-Mod$. Therefore, we can consider the short exact sequence 
$0\longrightarrow K\longrightarrow A\longrightarrow A/K\longrightarrow 0$ in $K-Mod$. Further, since 
$Spec(A)\longrightarrow Spec(K)$ is a Zariski open immersion, $A$ is a flat $K$-module. It follows that we have the following
short exact sequence in $K-Mod$:
\begin{equation}
0\longrightarrow A\otimes_KK\cong A\longrightarrow A\otimes_KA\longrightarrow A\otimes_K(A/K)\longrightarrow 0
\end{equation} Again, $K\longrightarrow A$ being an epimorphism in $Comm(\mathcal C)$, it follows
from Lemma \ref{Read1} that  $A\otimes_KK\cong A\cong A\otimes_KA$. Hence, $A\otimes_K(A/K)=0$. Using Lemma \ref{Lem4.4qg}, the monomorphism $K\longrightarrow A$ in $K-Mod$ splits and we can write $A\cong K\oplus K'$ for some $K'\in K-Mod$. Then:
\begin{equation}
(K\oplus K')\otimes_KK'\cong A\otimes_K(A/K)=0\qquad\Rightarrow\qquad 0=K\otimes_KK'\cong K'
\end{equation} Hence, $A\cong K$ and $i$ is an isomorphism. 

\medskip
In general, we consider a Zariski open immersion  $i:U\longrightarrow Spec(K)$. If $U=Spec(0)$, we are already done. Otherwise, we can  choose $V\in ZarAff(U)$ such that $V=Spec(A)$ with $A\ne 0$. From the above, it follows that
  $V\longrightarrow U\longrightarrow Spec(K)$ is an isomorphism. Then, the  following pullback square
\begin{equation}
\begin{CD}
V\cong U\times_UV\cong U\times_{Spec(K)}V@>>> V\cong Spec(K) \\
@VVV @V\cong VV \\
U @>i>> Spec(K)\\
\end{CD}
\end{equation} shows that $V\longrightarrow U$ is an isomorphism. This proves the result. 

\end{proof}

\medskip
We now come to the construction of field objects in $(\mathcal C,\otimes,1)$.
In general, for a commutative monoid $A\in Comm(\mathcal C)$, we let $\mathcal E(A)$ be the commutative ring
$\mathcal E(A):=Hom_{A-Mod}(A,A)$. Further, any morphism $f:A\longrightarrow B$ in $Comm(\mathcal C)$ induces a morphism
$\mathcal E(f):\mathcal E(A)\longrightarrow \mathcal E(B)$ of commutative rings. When $\mathcal E(A)$ is an integral domain and $A$ is Noetherian, we will
now localize $A$ with respect to the  non-zero elements
of $\mathcal E(A)$ (in the sense of \cite{AB4}) to obtain a field object $K(A)$. This is an analogue of the construction of the quotient field of an 
ordinary integral domain. 

\medskip
More precisely, for any $0\ne s\in \mathcal E(A)$, consider the localization $A_s$ of $A$ with respect to $s$ as  in 
\cite[(3.1)]{AB4}:
\begin{equation}\label{4.8zq}
A_s:=colim(A\overset{s}{\longrightarrow} A\overset{s}{\longrightarrow} A \overset{s}{\longrightarrow} \dots)
\end{equation} Further, by considering the morphisms $A_s\longrightarrow A_{st}$ for any
$s$, $t\in  \mathcal E(A)\backslash \{0\}$, we define $K(A)$ to be the colimit (see \cite[(3.3)]{AB4}):
\begin{equation}\label{4.8uq}
K(A):=\underset{s\in \mathcal E(A)\backslash \{0\}}{colim}\textrm{ }A_s
\end{equation}

\medskip
\begin{thm}\label{XABX} Let $A\in Comm(\mathcal C)$ be a Noetherian commutative monoid object such that $\mathcal E(A)$ is an integral domain. Let $K(A)$ be as defined in \eqref{4.8uq}. Then, 
$K(A)$ is a Noetherian monoid. \end{thm}

\begin{proof} In \cite[Proposition 3.2-3.3]{AB4}, it is already shown that the localization $K(A)$ is a commutative monoid
along with a morphism $i:A\longrightarrow K(A)$ in $Comm(\mathcal C)$ making $K(A)$ into a flat $A$-module. 

\medskip We want
to show that $i:A\longrightarrow K(A)$ is also an epimorphism in $Comm(\mathcal C)$. For this, we consider morphisms
$f,g:K(A)\longrightarrow B$ in $Comm(\mathcal C)$ with $h:=f\circ i=g\circ i:A\longrightarrow B$. Then,  we have $\mathcal E(h)
=\mathcal E(f)\circ \mathcal E(i)=\mathcal E(g)\circ \mathcal E(i)$.   Now, for any
$s\in \mathcal E(A)\backslash \{0\}$, $\mathcal E(i)(s)$ is a unit in $\mathcal E(K(A))$. Hence, $\mathcal E(h)(s)$ is a unit
in $\mathcal E(B)$. Then, from \cite[$\S$ 3]{AB4}, it follows that there exists a unique morphism  $j:K(A)\longrightarrow B$
in $Comm(\mathcal C)$ from the localization $K(A)$  such that $h=j\circ i$. Consequently, we have $f=j=g$. 

\medskip
Now since $i:A\longrightarrow K(A)$ is a flat epimorphism of commutative monoids and $A$ is Noetherian, it follows
from Proposition \ref{P4.2xr} that $K(A)$ is Noetherian. 
\end{proof} 

\medskip
\begin{thm}\label{Final4} Let $A\in Comm(\mathcal C)$ be a Noetherian commutative monoid object such that $\mathcal E(A)$ is an integral domain. Let $K(A)$ be as defined in \eqref{4.8uq}. Then, if $A$ is a finitely generated $A$-module, then $K(A)$ is a field object in $(\mathcal C,\otimes,1)$.  \end{thm}

\begin{proof} We have already checked that $K(A)$ is Noetherian. We will first show that $K(A)\ne 0$. We set $K:=K(A)$ and  
$S:=\mathcal E(A)\backslash \{0\}$. We choose some $s\in S$ and let $i:I\longrightarrow A$ be the kernel of $s:A\longrightarrow A$. Then, if
we consider any $f\in Hom_{A-Mod}(A,I)$, we see that $s\circ (i\circ f)=0$. Since $\mathcal E(A)$ is an integral domain and
$s\ne 0$, it follows that $i\circ f=0$. Further, since $i$ is a monomorphism, it follows that $f=0$. Therefore, we see that
$Hom_{A-Mod}(A,I)=0$. Since $A$ is Noetherian, it now follows as in the proof of Lemma \ref{Lem4.4qg} that 
$I=0$. Hence $s:A\longrightarrow A$ is a monomorphism for any $s\in S$. Now, considering the filtered
colimits appearing in \eqref{4.8zq} and \eqref{4.8uq}, it is clear that we have a monomorphism $A\longrightarrow K=K(A)$. Hence,
$K(A)\ne 0$. 

\medskip Since any $0\ne s:A\longrightarrow A$ is a monomorphism  and $A$ is finitely generated, it follows
from the filtered colimit of monomorphisms defining  $A_s$ in \eqref{4.8zq} that $\mathcal E(A_s)=Hom_{A_s-Mod}(A_s,A_s)\cong Hom_{A-Mod}(A,A_s)=\mathcal E(A)_s$. For any $0\ne t\in \mathcal E(A)$, the monomorphism $t:A\longrightarrow A$ induces
a monomorphism $t:A_s\longrightarrow A_s$ of filtered colimits.  Now the same
reasoning as in the previous paragraph shows that for any $0\ne s,t\in \mathcal E(A)$, 
the morphism $t:A_s\longrightarrow A_{st}$ is a monomorphism. Then, since $A$ is finitely generated, considering the filtered colimit
of monomorphisms defining $K(A)$ in \eqref{4.8uq}, we have $\mathcal E(K(A))=Hom_{K(A)-Mod}(K(A),K(A))\cong Hom_{A-Mod}(A,K(A))
=Q(\mathcal E(A))$ where $Q(\mathcal E(A))$ denotes the field of fractions of the integral domain $\mathcal E(A)$. 

\medskip On the other hand, let $i':I'\longrightarrow K=K(A)$ be a monomorphism
in $K-Mod$. Then, the induced map $i_K':Hom_{K-Mod}(K,I')\longrightarrow Hom_{K-Mod}(K,K)=\mathcal E(K)$ is a monomorphism of vector spaces over the field $\mathcal E(K)=Q(\mathcal E(A))$.  Hence,
$i_K'$ is either $0$ or an isomorphism. 
If $i_K':Hom_{K-Mod}(K,I')\longrightarrow Hom_{K-Mod}(K,K)$ is an isomorphism, it follows as in the proof of Proposition \ref{P4.5xpb}
that $i':I'\longrightarrow K$ is an isomorphism. On the other hand, if $i_K'=0$, it follows similarly that $I'=0$. 
 
\end{proof}

\medskip

\medskip
\section{Points of a Noetherian scheme   over a field object} 

\medskip

\medskip 
For an ordinary scheme $Z$ in usual algebraic geometry over $Spec(R)$, where $R$ is a ring, the  points of $Z$ over a field $K$ are the morphisms
$Spec(K)\longrightarrow Z$. Given a Noetherian, quasi-compact and semi-separated scheme $X$ over $(\mathcal C,\otimes,1)$, we will consider in this section
the   points of $X$ over a field object $K$ in $\mathcal C$. 
Accordingly, a  point of $X$ over $K$  corresponds to a   morphism 
$Spec(K)\longrightarrow X$ and therefore a pullback functor $QCoh(X)\longrightarrow QCoh(Spec(K))= K-Mod$.  In this section, we want
to find out which kinds of symmetric monoidal functors from $QCoh(X)$ to 
$K-Mod$  can be described as pullbacks by some morphism $Spec(K)\longrightarrow X$. 

\medskip
Let $A$ and $B$ be ordinary commutative rings and let $F:A-Mod\longrightarrow B-Mod$ be a cocontinuous  strong tensor functor from the category of $A$-modules to the category of $B$-modules. Then, $F$ induces a morphism $f:A=Hom_{A-Mod}(A,A)
\longrightarrow Hom_{B-Mod}(B,B)=B$ of rings. Then, it may be verified easily that the functor $F$ is given by ``extension of scalars'' 
along the morphism $f:A\longrightarrow B$. However, this argument does not extend directly to the case of commutative
monoid objects $A$ and $B$ in a symmetric monoidal category. In fact, in order to study functors between module categories over 
commutative monoid objects in $(\mathcal C,\otimes,1)$, we need the concept of ``normal functors'' from 
Vitale \cite{Vitale}. 

\medskip
\begin{defn}(see \cite[Definition 4.2]{Vitale}) \label{Normfunc} Let $A$ and $B$ be commutative monoid objects in $(\mathcal C,\otimes,1)$. A normal functor
$(F,\tau):A-Mod\longrightarrow B-Mod$ consists of a functor $F:A-Mod\longrightarrow B-Mod$ along with a family of morphisms:
\begin{equation}
\tau_{M,N}:\underline{Hom}_A(M,N)\longrightarrow \underline{Hom}_B(F(M),F(N)) 
\qquad\forall\textrm{ }M,N\in A-Mod
\end{equation} in $\mathcal C$  satisfying the following three conditions:

\medskip
(1) The family $\{\tau_{M,N}\}_{M,N\in A-Mod}$ is natural in $M$ and $N$. 

\medskip
(2) The family $\{\tau_{M,N}\}_{M,N\in A-Mod}$  is compatible with internal composition, i.e., for $M$, $N$, $P\in A-Mod$, we have a commutative diagram:
\begin{equation}
\begin{CD}
\underline{Hom}_A(M,N)\otimes \underline{Hom}_A(N,P) @>\tau_{M,N}\otimes \tau_{N,P}>> 
\underline{Hom}_B(F(M),F(N))\otimes \underline{Hom}_B(F(N),F(P)) \\ 
@V\circ VV @V\circ VV \\
\underline{Hom}_A(M,P) @>\tau_{M,P}>> \underline{Hom}_B(F(M),F(P))\\ 
\end{CD}
\end{equation} Further, the family is also compatible with the identity $1\longrightarrow A\longrightarrow \underline{Hom}_A(M,M)$
of the monoid $\underline{Hom}_A(M,M)$ for any $M\in A-Mod$.

\medskip
(3) For any $f:M\longrightarrow N$ in $A-Mod$, the following diagram commutes:
 \begin{equation}
\begin{CD}
1@>1>> 1 \\
@V\widetilde{f}VV @V\widetilde{F(f)}VV \\
\underline{Hom}_A(M,N) @>\tau_{M,N}>> \underline{Hom}_B(F(M),F(N))\\ 
\end{CD}
\end{equation} where $\widetilde{f}:1\longrightarrow \underline{Hom}_A(M,N)$ (resp. $\widetilde{F(f)}:1
\longrightarrow  \underline{Hom}_B(F(M),F(N))$) is the factorization of $f$ (resp. $F(f)$) through
$\underline{Hom}_A(M,N)\longrightarrow \underline{Hom}(M,N)$ (resp. $\underline{Hom}_B(F(M),F(N))\longrightarrow \underline{Hom}(F(M),F(N))$). 
\end{defn}

\medskip
\begin{thm}\label{TYpcf} Let $A$, $B$ be commutative monoid objects in $(\mathcal C,\otimes,1)$ and suppose that
$A$ is Noetherian. Let
$(F,\tau):A-Mod\longrightarrow B-Mod$ be a normal functor such that $F:A-Mod\longrightarrow B-Mod$ is a cocontinuous
symmetric monoidal functor. Then, there is a morphism $f:A\longrightarrow B$ in $Comm(\mathcal C)$ such that 
$F(M)=B\otimes_AM$ for each $M\in B-Mod$. 
\end{thm} 

\begin{proof} We set $f:=\tau_{A,A}:A=\underline{Hom}_A(A,A)\longrightarrow 
\underline{Hom}_B(B,B)=B$. From the conditions in Definition \ref{Normfunc}, we see that $f$ is a morphism in
$Comm(\mathcal C)$. Since $F$ is a symmetric monoidal functor, we see that $F(A)=B=B\otimes_AA$. Now, if $M$ is a finitely
generated $A$-module, we can express $M$ as a colimit $M=colim(0\longleftarrow A^m \longrightarrow A^n)$. Then, since
$F$ is cocontinuous, we have:
\begin{equation}\label{5.4pmn}
F(M)=colim(0\longleftarrow F(A^m)\longrightarrow F(A^n))=colim(0\longleftarrow B\otimes_AA^m
\longrightarrow B\otimes_AA^n)=B\otimes_AM
\end{equation} Finally, since any module in $A-Mod$ can be expressed as a colimit of finitely generated submodules,
it follows from \eqref{5.4pmn} that $F(M)=B\otimes_AM$ for any $M\in A-Mod$. 
\end{proof}

 \medskip 
Let $X$ be a Noetherian, quasi-compact and semi-separated scheme  over $(\mathcal C,\otimes,1)$ and let $K$ be a 
field object. The result of Proposition \ref{TYpcf} suggests that morphisms $Spec(K)\longrightarrow X$ should correspond to ``normal functors'' from $QCoh(X)$
to $K-Mod$ that are also cocontinuous and symmetric monoidal. However, it is not immediately clear how we can define the notion of a ``normal
functor'' from $QCoh(X)$
to $K-Mod$. In order to introduce this notion, we will make use of the following definition from \cite{BC}, which explains
what it means for a functor from $QCoh(X)$
to $K-Mod$ to be ``local'' with respect to some $U\in ZarAff(X)$. 
 
 \medskip
 \begin{defn}\label{Def4.4} (see \cite[Definition 2.3.5]{BC}) Let $(C,\otimes,1_C)$, $(D,\otimes,1_D)$ and $(E,\otimes,1_E)$ be   preadditive symmetric
 monoidal categories. Let $(i^*,i_*)$ be a pair of adjoint functors between $E$ and $D$, with $i^*:E\longrightarrow D$  a symmetric monoidal
 functor and $i_*:D\longrightarrow E$ a lax tensor functor. Further, suppose that the counit $\varepsilon: i^*i_*\longrightarrow id_D$ is an isomorphism. Let
 $\eta: id_E\longrightarrow i_*i^*$ be the unit natural transformation. 
 Then, a   functor $F:E\longrightarrow C$ 
 is said to be $i$-local if $F\eta :F\longrightarrow Fi_*i^*$ is an isomorphism. 
  \end{defn}

\medskip
From Definition \ref{Def4.4}, it may be easily verified that we have an equivalence of categories:
\begin{equation}
Fun(D,C)\doublerightleftarrow{\circ i^*}{\circ i_*}\{F\in Fun(E,C), \textrm{ }\mbox{$i$-local}\}
\end{equation} Here $Fun(D,C)$ (resp. $Fun(E,C)$) denotes the category of functors from $D$ (resp. from $E$)
to $C$. We will also need the following result from \cite{BC}.

  \medskip
  \begin{thm}\label{P4.6fp} (see \cite[Proposition 2.3.6]{BC}) In the setup of Definition \ref{Def4.4}, one has: 
  
  \medskip
  (a) Let $f$ be any morphism in $E$ such that $i^*(f)$ is an isomorphism in $D$. Then, a  functor $F:E\longrightarrow C$ is $i$-local if and only if $F(f)$ is an isomorphism in $C$ for every such morphism $f$. 
  
  \medskip
  (b) Let $Fun_{c\otimes}(D,C)$ (resp. $Fun_{c\otimes}(E,C)$ ) be the category of all cocontinuous symmetric
  monoidal functors from $D$ (resp. $E$) to $C$. Then, we have an equivalence of categories: 
  \begin{equation}
  Fun_{c\otimes}(D,C)\doublerightleftarrow{\circ i^*}{\circ i_*} \{F\in Fun_{c\otimes}(E,C), \textrm{ }\mbox{$i$-local}\}
  \end{equation} 
  
  \end{thm}

\medskip
 In particular, let $X$ be a Noetherian, quasi-compact and semi-separated scheme over $(\mathcal C,\otimes,1)$ and let $i:U\longrightarrow X$
  be a Zariski open immersion with $U$ affine. Then, as mentioned in Section 2, it follows that $i^*:QCoh(X)\longrightarrow QCoh(U)$
  is a symmetric monoidal functor and $i_*:QCoh(U)\longrightarrow QCoh(X)$ is a lax tensor functor. Further, we know from
  Proposition \ref{P2.5qp} that the counit $i^*i_*\longrightarrow id_{QCoh(U)}$ is an isomorphism of functors. Therefore, 
  in particular, we can set $D=QCoh(U)$ and $E=QCoh(X)$ in Definition \ref{Def4.4}. Further, given a commutative monoid object $A$ in $(\mathcal C,\otimes,1)$, we 
  can set $C=A-Mod$ in Definition \ref{Def4.4}. Accordingly, we will say that a   functor 
$F:QCoh(X)\longrightarrow A-Mod$ is $U$-local if the natural transformation $F\longrightarrow  F\circ i_*\circ i^*$ is an isomorphism. 
We are now ready to introduce normal functors from $QCoh(X)$ to $A-Mod$ for some $A\in Comm(\mathcal C)$. 

\medskip
\begin{defn}\label{NDef} Let $X$ be a Noetherian, quasi-compact and semi-separated scheme over $(\mathcal C,\otimes,1)$. Let $A$ be a commutative monoid object in $\mathcal C$. Then, we will say that a functor $F:QCoh(X)\longrightarrow A-Mod$ is  normal if there   exists a  normal functor $(G^U,\tau^U):B-Mod\longrightarrow A-Mod$ with $F\cong G^U\circ i^*$
for each $(i:U=Spec(B)\longrightarrow X)\in ZarAff(X)$ such that $F$ is $i$-local.
\end{defn}

\medskip
We remark that the condition in Definition \ref{NDef} could be satisfied vacuously, i.e., there might not exist 
$U\in ZarAff(X)$ such that $F$ is $U$-local. 

  \medskip 
Let $K$ be a field object in $(\mathcal C,\otimes,1)$. We will now show that morphisms $Spec(K)\longrightarrow X$ correspond to cocontinuous symmetric monoidal normal 
functors from $QCoh(X)$ to $K-Mod$. For this, the first step is to show  that  given a cocontinuous symmetric monoidal functor $F:QCoh(X)\longrightarrow K-Mod$, there exists 
$U\in ZarAff(X)$ such that $F$ is $U$-local. The latter is an analogue of the result of Brandenburg and Chirvasitu \cite[Lemma 3.3.6]{BC}
in the case of usual schemes. 

\medskip
Let $\mathcal M$ be a quasi-coherent sheaf on $X$ and let $\mathcal N$, $\mathcal N'$ be quasi-coherent submodules
of $\mathcal M$. For any $U\in ZarAff(X)$, consider the morphism:
\begin{equation}
\mathcal O_X(U)\longrightarrow \underline{Hom}_{\mathcal O_X(U)}(\mathcal N(U),(\mathcal M/\mathcal N')(U)) 
\end{equation} that corresponds, by adjointness, to the composition 
$\mathcal O_X(U)\otimes_{\mathcal O_X(U)}\mathcal N(U)\longrightarrow \mathcal M(U)\longrightarrow (\mathcal M/\mathcal N')(U)$ in $\mathcal O_X(U)-Mod$. We set:
\begin{equation}\label{4.3gh}
U\mapsto (\mathcal N':\mathcal N)(U):=Ker(\mathcal O_X(U)\longrightarrow \underline{Hom}_{\mathcal O_X(U)}(\mathcal N(U),(\mathcal M/\mathcal N')(U)) )
\end{equation}

\medskip
\begin{thm}\label{A380} Let $X$ be a Noetherian, quasi-compact and semi-separated scheme over $(\mathcal C,\otimes,1)$. 
Let $\mathcal M$ be a quasi-coherent sheaf on $X$ and let $\mathcal N$, $\mathcal N'$ be quasi-coherent submodules
of $\mathcal M$. 
Let $[\mathcal N':\mathcal N]$ be the set of all
quasi-coherent submodules $\mathcal I$ of $\mathcal O_X$ such that the composed morphism:
\begin{equation}\label{4.4hg}
\mathcal I\otimes_{\mathcal O_X} \mathcal N\longrightarrow \mathcal O_X\otimes_{\mathcal O_X}\mathcal N\longrightarrow \mathcal M
\end{equation} factors through $\mathcal N'$. Then,  if $\mathcal N$ is finitely generated, we have:

\medskip
(a) The association $
U\mapsto (\mathcal N':\mathcal N)(U)$ for all $U\in ZarAff(X)
$ defines a quasi-coherent submodule of  $\mathcal O_X$. 

\medskip
(b) The quasi-coherent submodule $(\mathcal N':\mathcal N)$ equals the sum of all the submodules ${\mathcal I\in [\mathcal N':\mathcal N]} $. 
 \end{thm}
 
 \begin{proof} (a) Let $U\in ZarAff(X)$ and let $V\in ZarAff(U)$. We need to show that
 $(\mathcal N':\mathcal N)(V)\cong (\mathcal N':\mathcal N)(U)\otimes_{\mathcal O_X(U)}\mathcal O_X(V)$. Since $\mathcal N$ is finitely
 generated and $X$ is Noetherian, it follows from Definition \ref{Def4.1} that there exists a morphism $\mathcal O_X(U)^m\longrightarrow \mathcal O_X(U)^n$ such that:
 \begin{equation}\label{4.5gh}
 \begin{array}{c}
 \mathcal N(U)=colim(0\longleftarrow \mathcal O_X(U)^m\longrightarrow \mathcal O_X(U)^n)\\
 \mathcal N(V)\cong \mathcal N(U)\otimes_{\mathcal O_X(U)}\mathcal O_X(V)=colim(0\longleftarrow \mathcal O_X(V)^m
 \longrightarrow \mathcal O_X(V)^n)\\
 \end{array}
 \end{equation} From \eqref{4.5gh}, it follows that:
 \begin{equation}\label{4.6gh}
 \begin{array}{l}
 \underline{Hom}_{\mathcal O_X(U)}(\mathcal N(U),(\mathcal M/\mathcal N')(U)))\otimes_{\mathcal O_X(U)}\mathcal O_X(V) \\
 \cong lim(0\longrightarrow (\mathcal M/\mathcal N')(U))^m\longleftarrow (\mathcal M/\mathcal N')(U))^n)\otimes_{\mathcal O_X(U)}\mathcal O_X(V)\\
 \cong lim(0\longrightarrow (\mathcal M/\mathcal N')(V))^m\longleftarrow (\mathcal M/\mathcal N')(V))^n) \\
\cong  \underline{Hom}_{\mathcal O_X(V)}(\mathcal N(V),(\mathcal M/\mathcal N')(V)))
 \end{array}
 \end{equation} Now, since the kernel defining  $(\mathcal N':\mathcal N)(U)$ in \eqref{4.3gh} is a finite limit, it follows from \eqref{4.6gh} that
 we have $(\mathcal N':\mathcal N)(V)\cong (\mathcal N':\mathcal N)(U)\otimes_{\mathcal O_X(U)}\mathcal O_X(V)$. 
 
 \medskip
 (b) From part (a), we know 
that $(\mathcal N':\mathcal N)$ is a quasi-coherent submodule of $\mathcal O_X$. For  a quasi-coherent submodule $\mathcal I$ of $\mathcal O_X$, the morphism in   \eqref{4.4hg} factors through $\mathcal N'$
 if and only if the composition $\mathcal I\otimes_{\mathcal O_X} \mathcal N\longrightarrow \mathcal O_X\otimes_{\mathcal O_X}\mathcal N\longrightarrow \mathcal M
 \longrightarrow \mathcal M/\mathcal N'$ is $0$.  Then, using adjointness, it follows that 
 $\mathcal I\in [\mathcal N':\mathcal N]$ if and only for each $U\in ZarAff(X)$,  the composition $\mathcal I(U)\longrightarrow \mathcal O_X(U)
 \longrightarrow \underline{Hom}_{\mathcal O_X(U)}(\mathcal N(U),
 (\mathcal M/\mathcal N')(U))$ is $0$.   By definition, we have:
 \begin{equation} \label{4.7vy}
  (\mathcal N':\mathcal N)(U):=Ker(\mathcal O_X(U)\longrightarrow \underline{Hom}_{\mathcal O_X(U)}(\mathcal N(U),(\mathcal M/\mathcal N')(U)) )
\end{equation} and hence any $\mathcal I\in [\mathcal N':\mathcal N]$ must be a submodule of $(\mathcal N':\mathcal N)$.  In particular, from   
\eqref{4.7vy} it is also clear that $(\mathcal N':\mathcal N)\in [\mathcal N':\mathcal N]$. Thus, given any quasi-coherent submodule $\mathcal I'$ of 
$\mathcal O_X$ such that $\mathcal I\subseteq \mathcal I'$ $\forall$ $\mathcal I\in [\mathcal N':\mathcal N]$, we have
$(\mathcal N':\mathcal N)\subseteq \mathcal I'$. Hence $(\mathcal N':\mathcal N)$ is the sum of all submodules $\mathcal I\in [\mathcal N':\mathcal N]$. 
 \end{proof}

\medskip
\begin{lem}\label{Lem4.5} Let $X$ be a Noetherian, quasi-compact and semi-separated scheme over $(\mathcal C,\otimes,1)$ and let 
$\mathcal I$ be a quasi-coherent submodule of $\mathcal O_X$. Let $K$ be a field object in $(\mathcal C,\otimes,1)$. 
Let $F$ be a symmetric monoidal functor from
$QCoh(X)$ to $K-Mod$. Then, if the induced map $F(\mathcal I\longrightarrow \mathcal O_X)$ is an epimorphism
in $K-Mod$, it must be an isomorphism. 
\end{lem}

\begin{proof} Since $F$ is a symmetric monoidal functor, we know that $F(\mathcal O_X)=K$. We set $L:=F(\mathcal I)$. Then, it 
is clear that $L$ is a (not necessarily unital) commutative monoid object in $K-Mod$ and we have an epimorphism
$f:L\longrightarrow K$ in $K-Mod$. It follows from Lemma \ref{Lem5.4lg} that  
$Hom_{K-Mod}(K,f):Hom_{K-Mod}(K,L)\longrightarrow Hom_{K-Mod}(K,K)$ is a surjection. We choose any 
$u:K\longrightarrow L$ such that  $u$ is mapped to the identity map $K\longrightarrow K$ by this surjection $Hom_{K-Mod}(K,f)$. We now consider the following
two commutative diagrams: 

\begin{equation*}
\begin{tikzpicture}
      \matrix[matrix of math nodes,column sep=5pc,row sep=3em]
      {
                         |(A)| \mathcal I     &          \\
        |(C)| \mathcal I\otimes_{\mathcal O_X} \mathcal I & |(D)| \mathcal O_X\otimes_{\mathcal O_X} \mathcal I     \\
      };
      \begin{scope}[->]
             \draw (C) -- (A);
         \draw (C) -- (D);
         \draw (D) -- (A);
      \end{scope}
   \end{tikzpicture} \qquad\qquad 
 \begin{tikzpicture}
      \matrix[matrix of math nodes,column sep=5pc,row sep=3em]
      {
                         |(A)| L &              \\
        |(C)| L\otimes_K L & |(D)| K\otimes_K L     \\
      };
      \begin{scope}[->]
             \draw (C) -- (A);
         \draw (C) -- (D);
         \draw (D) -- (A);
      \end{scope}
   \end{tikzpicture}
   \end{equation*} where the diagram on the right is obtained by applying the symmetric monoidal functor $F$. Since
   $f\circ u=id$   as mentioned above, the composition
   $\begin{CD}K\otimes_KL@>u\otimes_K 1_L>> L\otimes_KL@>f\otimes_K 1_L>>K\otimes_KL\end{CD}$ is also the identity. Then 
   we have the following commutative diagram:
   \begin{equation}\label{4.8eqp}
  \begin{tikzpicture}[baseline=(current  bounding  box.center)]
      \matrix[matrix of math nodes,column sep=5pc,row sep=3em]
      {
                        & |(A)| L      &         \\
        |(C)| K\otimes_K L & |(D)| L\otimes_K L   & |(E)| K\otimes_KL  \\
      };
      \begin{scope}[->]
             \draw (C) to node {$\cong\qquad$}  (A);
         \draw  (C) edge node [below] {$u\otimes_K1_L$} (D);
         \draw (D) -- (A);
         \draw (D) edge node [below] {$f\otimes_K1_L$}  (E);
         \draw (E) to node {$\qquad\cong$} (A);
      \end{scope}
   \end{tikzpicture}
   \end{equation} From \eqref{4.8eqp}, it follows that $L$ is actually a unital commutative monoid in $K-Mod$ with unit
   morphism $u:K\longrightarrow L$. Therefore,  the following composition is identical to the isomorphism
   $L\otimes_KK\cong L$:
   \begin{equation}\label{4.9qc}
   \begin{CD}
   L\otimes_KK @>1_L\otimes_Ku>> L\otimes_KL@>f\otimes_K1_L>> K\otimes_KL@>\cong>> L
   \end{CD}
   \end{equation} Since $(1_K\otimes_Ku)\circ (f\otimes_K1_K)=(f\otimes_K1_L)\circ (1_L\otimes_Ku)$, it follows from \eqref{4.9qc}
   that $f\otimes_K1_K:L\otimes_KK\longrightarrow K\otimes_KK$ is a monomorphism. Then, the morphism $\begin{CD}f :L\cong L\otimes_KK @>f\otimes_K1_K>>K\otimes_KK\cong K
   \end{CD}$ is also a monomorphism. Since $f$ is already
   an epimorphism, we now know that $f$ is an isomorphism. 
\end{proof} 

\medskip
\begin{thm}\label{thm4.6c} Let $X$ be a Noetherian, quasi-compact and semi-separated scheme over $(\mathcal C,\otimes,1)$ and let 
 $K$ be a field object in $(\mathcal C,\otimes,1)$. 
Let $F$ be  a cocontinuous symmetric monoidal functor from
$QCoh(X)$ to $K-Mod$. Then, there exists $U\in ZarAff(X)$ such that if $\mathcal I$ is a quasi-coherent submodule
of $\mathcal O_X$ with $\mathcal I|_U=\mathcal O_X|_U$, the functor  $F$ maps the monomorphism 
$\mathcal I\longrightarrow \mathcal O_X$ to an isomorphism. 
\end{thm} 

\begin{proof} Let $\{U_i\}$, $1\leq i\leq n$ be an affine Zariski cover of $X$. Suppose that for each $i$, there exists a quasi-coherent
submodule $\mathcal I_i$ of $\mathcal O_X$ such that $\mathcal I_i|_{U_i}=\mathcal O_X|_{U_i}$ and $F(\mathcal I_i)\longrightarrow
F(\mathcal O_X)=K$ is not an epimorphism. Since $K$ is a field object, it follows that $F(\mathcal I_i)=0$. We now consider
the direct sum $\bigoplus_{i=1}^n\mathcal I_i$ and the sum $\sum_{i=1}^n\mathcal I_i\hookrightarrow \mathcal O_X$. Since each
$F(\mathcal I_i)=0$ and the morphism $\bigoplus_{i=1}^n\mathcal I_i\longrightarrow 
\mathcal O_X$ factors through $\sum_{i=1}^n\mathcal I_i$, we see that the composition
 $F(\bigoplus_{i=1}^n\mathcal I_i)\longrightarrow F(\sum_{i=1}^n\mathcal I_i)\longrightarrow F(\mathcal O_X)=K$ is $0$.  
Further, from the proof of Proposition \ref{P2.6}, we know that $\bigoplus_{i=1}^n\mathcal I_i\longrightarrow
 \sum_{i=1}^n\mathcal I_i$ is an epimorphism. Since $F$ preserves small colimits (and hence preserves cokernels), 
 $F(\bigoplus_{i=1}^n\mathcal I_i)\longrightarrow F(\sum_{i=1}^n\mathcal I_i)$ is an epimorphism. It now follows that
 $F(\sum_{i=1}^n\mathcal I_i)\longrightarrow F(\mathcal O_X)=K$ is $0$. 
 
 \medskip
 On the other hand, we see that for any given $1\leq j\leq n$, the composition $\mathcal O_X|_{U_j}=\mathcal I_j|_{U_j}\longrightarrow
(\sum_{i=1}^n\mathcal I_i)|_{U_j}\longrightarrow \mathcal O_X|_{U_j}$ is an isomorphism. Hence, $(\sum_{i=1}^n\mathcal I_i)|_{U_j}\longrightarrow \mathcal O_X|_{U_j}$
is an epimorphism for each $j$. Since the $\{U_i\}_{1\leq i\leq n}$ form a cover of $X$, it follows that
$\sum_{i=1}^n\mathcal I_i\longrightarrow \mathcal O_X$ is also an epimorphism and therefore $\sum_{i=1}^n\mathcal I_i\overset{\cong}{\longrightarrow}\mathcal O_X$. 
This contradicts the fact that the morphism $F(\sum_{i=1}^n\mathcal I_i)\longrightarrow F(\mathcal O_X)=K$ is $0$. Therefore, there exists
at least one $1\leq i_0\leq n$ such that for any quasi-coherent submodule $\mathcal I$ of $\mathcal O_X$ with $\mathcal I|_{U_{i_0}}
=\mathcal O_X|_{U_{i_0}}$, $F(\mathcal I\longrightarrow \mathcal O_X)$ is always an epimorphism. Combining with the result of 
Lemma \ref{Lem4.5}, we see that $F(\mathcal I\longrightarrow \mathcal O_X)$ is actually an isomorphism for any such
$\mathcal I$. 

\end{proof}

\begin{lem}\label{Lem4.08} Let $\mathcal N$ be a quasi-coherent sheaf on $X$ and let $\mathcal M$, $\mathcal M_1$ and
$\mathcal M_2$ be  quasi-coherent submodules of $\mathcal N$. Let $\mathcal I$ 
be a quasi-coherent submodule of $\mathcal O_X$ such that the morphisms
$\mathcal I\otimes_{\mathcal O_X}\mathcal M_1\longrightarrow \mathcal N$, $\mathcal I\otimes_{\mathcal O_X}\mathcal M_2
\longrightarrow \mathcal N$ factor through $\mathcal M$. Then, the morphism
$\mathcal I\otimes_{\mathcal O_X} (\mathcal M_1+\mathcal M_2)\longrightarrow \mathcal N$ factors through
$\mathcal M$.  
\end{lem}
\begin{proof} We choose $Spec(A)=U\in ZarAff(X)$ and set:
\begin{equation}
I:=\mathcal I(U)\qquad M_1=\mathcal M_1(U) \qquad M_2:=\mathcal M_2(U) \qquad
M:=\mathcal M(U) \qquad N:=\mathcal N(U)
\end{equation} Then, the morphisms $I\otimes_A M_1\longrightarrow N$, $I\otimes_A M_2\longrightarrow N$ factor through
$M$. We will show that $I\otimes_A (M_1+M_2)\longrightarrow N$ factors through $M$. We now consider the composition
\begin{equation}\label{4.15su}
Ker(I\otimes_A (M_1\oplus M_2)\longrightarrow I\otimes_A N)\longrightarrow I\otimes_A (M_1\oplus M_2)\longrightarrow M\longrightarrow N
\end{equation} Since the composition $I\otimes_A (M_1\oplus M_2)\longrightarrow M\longrightarrow N$ coincides
with $I\otimes_A (M_1\oplus M_2)\longrightarrow I\otimes_A N\longrightarrow N$, the composition in \eqref{4.15su} is $0$. 
Since $M\longrightarrow N$ is a monomorphism, this implies that
\begin{equation}\label{4.16su}
Ker(I\otimes_A (M_1\oplus M_2)\longrightarrow I\otimes_A N)\longrightarrow I\otimes_A (M_1\oplus M_2)\longrightarrow M 
\end{equation} is $0$. Composing with the canonical morphism $I\otimes_A Ker((M_1\oplus M_2)\longrightarrow N)
\longrightarrow Ker(I\otimes_A (M_1\oplus M_2)\longrightarrow I\otimes_A N)$, we see that
\begin{equation}\label{4.17su}
I\otimes_A Ker((M_1\oplus M_2)\longrightarrow N)\longrightarrow I\otimes_A (M_1\oplus M_2)\longrightarrow M
\end{equation} is $0$. Then, \eqref{4.17su} implies that there is a natural morphism
\begin{equation}
\begin{array}{l}
I\otimes_A (M_1+M_2)\cong I\otimes_A Coker(Ker((M_1\oplus M_2)\longrightarrow N)\longrightarrow M_1\oplus M_2)\\
\qquad\qquad \overset{\cong}{\longrightarrow}
Coker(I\otimes_A Ker((M_1\oplus M_2)\longrightarrow N)\longrightarrow I\otimes_A (M_1\oplus M_2))\longrightarrow M \\
\end{array}
\end{equation}
\end{proof}

\medskip
\begin{thm}\label{P4.10x} Let $X$ be a Noetherian, quasi-compact and semi-separated scheme over $(\mathcal C,\otimes,1)$ and let 
 $K$ be a field object. 
Let $F$ be  a cocontinuous symmetric monoidal functor from
$QCoh(X)$ to $K-Mod$.  Let  $U\in ZarAff(X)$ be such that for any  quasi-coherent submodule
$\mathcal I$ of $\mathcal O_X$ with $\mathcal I|_U=\mathcal O_X|_U$,  $F$ maps the monomorphism 
$\mathcal I\longrightarrow \mathcal O_X$ to an isomorphism. Then, if $g:\mathcal M\longrightarrow 
\mathcal N$ is a monomorphism in $QCoh(X)$ such that $g|_U:\mathcal M|_U
\longrightarrow \mathcal N|_U$ is an isomorphism, the functor $F$ maps
$g$ to an isomorphism $F(g):F(\mathcal M)\overset{\cong}{\longrightarrow} F(\mathcal N)$. 
\end{thm}

\begin{proof} From Theorem \ref{Thm3.8}, it follows that $\mathcal N$ can be expressed as a 
filtered direct limit of its finitely generated quasi-coherent submodules $\{\mathcal N_i\}_{i\in I}$. 
 Since each $\mathcal N_i$ is finitely generated, it follows from 
Proposition \ref{A380}(a) that we can consider  
the quasi-coherent submodule $\mathcal I_i:=(\mathcal M:\mathcal N_i)$ of $\mathcal O_X$ for each $i$. Further, for any quasi-coherent submodule $\mathcal I$ of $\mathcal O_X$, it  follows from Lemma \ref{Lem4.08}
that the morphism
$\mathcal I\otimes_{\mathcal O_X}\mathcal N_i\longrightarrow \mathcal N$ factors through $\mathcal M$ if and only if
$\mathcal I\otimes_{\mathcal O_X} (\mathcal M+\mathcal N_i)\longrightarrow \mathcal N$ also factors through $\mathcal M$. In particular, this means
that $\mathcal I_i\otimes_{\mathcal O_X}(\mathcal M+\mathcal N_i)\longrightarrow \mathcal N$ factors through
$\mathcal M$.  

\medskip  For the sake of convenience, 
we set $\mathcal N_i':=\mathcal M+\mathcal N_i$. It is clear that $\underset{i\in I}{\varinjlim}\textrm{ }
\mathcal N_i'=\mathcal N$. Further, since $g|_U:\mathcal M|_U\longrightarrow \mathcal N|_U$ is an isomorphism, 
we have $\mathcal I_i|_U=(\mathcal M|_U:\mathcal N_i|_U)=\mathcal O_X|_U$.  Hence $F(\mathcal I_i)\cong 
F(\mathcal O_X)=K\in K-Mod$. We have noted before that  $\mathcal I_i\otimes_{\mathcal O_X}
\mathcal N_i'=\mathcal I_i\otimes_{\mathcal O_X}(\mathcal M+\mathcal N_i)\longrightarrow \mathcal N$ factors through
$\mathcal M$. Since $F$ is a
symmetric monoidal functor, we now have the following commutative diagrams: 
\begin{equation}\label{4.15a380}
\begin{CD}
\mathcal I_i\otimes_{\mathcal O_X} \mathcal M @>>> \mathcal O_X\otimes_{\mathcal O_X} \mathcal M \\
@VVV @V\cong VV \\
\mathcal I_i\otimes_{\mathcal O_X} \mathcal N_i' @>>> \mathcal M \\
\end{CD} \qquad \Rightarrow \qquad
\begin{CD}
F(\mathcal I_i)\otimes_K F(\mathcal M) 
\cong F(\mathcal M) @>\cong >>  F(\mathcal M) \\
@VVV @Vid VV \\
F(\mathcal I_i)\otimes_KF(\mathcal N_i')\cong F(\mathcal N_i') @>>> F(\mathcal M) \\
\end{CD}
\end{equation} From the right hand side diagram in \eqref{4.15a380}, it follows that
$F(\mathcal M)\longrightarrow F(\mathcal N_i')$ is a monomorphism for each $i\in I$. On the other
hand, we also have:
\begin{equation}\label{4.16a380}
\begin{CD}
\mathcal I_i\otimes_{\mathcal O_X}\mathcal N_i' @>>> \mathcal O_X\otimes_{\mathcal O_X} \mathcal N_i'\\
@VVV @V\cong VV\\
\mathcal M @>>> \mathcal N_i' \\
\end{CD} \qquad\Rightarrow\qquad
\begin{CD}
F(\mathcal I_i)\otimes_KF(\mathcal N_i')\cong F(\mathcal N_i') @>\cong >> F(\mathcal N_i') \\
@VVV @Vid VV\\
F(\mathcal M)@>>> F(\mathcal N_i') \\
\end{CD}
\end{equation} From the right hand side diagram in \eqref{4.16a380}, it follows that
$F(\mathcal M)\longrightarrow F(\mathcal N_i')$ is also an epimorphism for each $i\in I$. Consequently,
we have an isomorphism $F(\mathcal M)\overset{\cong}{\longrightarrow} F(\mathcal N_i')$  for each $i\in I$. 
Since $F$ preserves colimits, it now follows that $F(\mathcal M)\cong \underset{i\in I}{\varinjlim}\textrm{ }
F(\mathcal N_i')\cong F( \underset{i\in I}{\varinjlim}\textrm{ }\mathcal N_i')\cong F(\mathcal N)$. 

\end{proof}

\medskip
\begin{thm}\label{P5.12wvo} Let $X$ be a quasi-compact, semi-separated and Noetherian scheme over $(\mathcal C,\otimes,1)$. 
Let $K$ be a field object of $(\mathcal C,\otimes,1)$ and $F:QCoh(X)\longrightarrow K-Mod$ be a cocontinuous
symmetric monoidal functor. Then, there exists a Zariski open immersion $i:U\longrightarrow X$ with $U$ affine such that
$F$ is $i$-local.
\end{thm} 

\begin{proof} Using Proposition \ref{thm4.6c}, we can choose  $i:U\longrightarrow X$ in $ZarAff(X)$ such that if $\mathcal I$ is a quasi-coherent submodule
of $\mathcal O_X$ with $\mathcal I|_U=\mathcal O_X|_U$, then $F$ maps the monomorphism 
$\mathcal I\longrightarrow \mathcal O_X$ to an isomorphism. We now consider a morphism $g:\mathcal M\longrightarrow \mathcal N$ in 
$QCoh(X)$ such that $g|_U:\mathcal M|_U\longrightarrow \mathcal N|_U$ is an isomorphism. Then, in particular, if $g:\mathcal M
\longrightarrow \mathcal N$ is a monomorphism, $F(g):F(\mathcal M)\longrightarrow F(\mathcal N)$ is an isomorphism as shown
in Proposition \ref{P4.10x}

\medskip
Alternatively, suppose that $g:\mathcal M\longrightarrow \mathcal N$ is an epimorphism such that $g|_U$ is an isomorphism. 
We now consider the commutative diagram
\begin{equation}
 \begin{tikzpicture}[baseline=(current  bounding  box.center)]
      \matrix[matrix of math nodes,column sep=5pc,row sep=3em]
      {
                        & |(A)| \mathcal M      &         \\
        |(C)| \mathcal M & |(D)| \mathcal M':=lim(\mathcal M\overset{g}{\longrightarrow}\mathcal N\overset{g}{\longleftarrow}
\mathcal M)   & |(E)| \mathcal M  \\
      };
      \begin{scope}[->]
             \draw (A) to node {$id\qquad\qquad$}  (C);
         \draw  (D) edge node [below] {$p_1$} (C);
         \draw (A) edge node [right] {$i$} (D);
         \draw (D) edge node [below] {$p_2$} (E);
         \draw (A) to node {$\qquad\qquad id$} (E);
      \end{scope}
   \end{tikzpicture}
\end{equation} where  $p_1$, $p_2$ are the two canonical morphisms from the limit $
\mathcal M'$  to $\mathcal M$. Since $p_1\circ i=p_2\circ i=id$, $i:\mathcal M\longrightarrow \mathcal M'$ is a
monomorphism. When restricted to $U$, $g$ becomes an isomorphism and hence the restriction $i|_U$ 
is an isomorphism. From Proposition \ref{P4.10x}, it follows that $F(i):F(\mathcal M)
\longrightarrow F(\mathcal M')$  
is actually an isomorphism. Then, since $p_1\circ i=p_2\circ i=id$, $F(p_1)$ and 
$F(p_2)$ are isomorphisms. We now consider the commutative diagrams:
\begin{equation}\label{4.24e}
\begin{CD}
\mathcal M' @>p_2>> \mathcal M \\
@Vp_1VV @VgVV \\
\mathcal M @>g>> \mathcal N \\
\end{CD} \qquad\Rightarrow \qquad 
\begin{CD}
F(\mathcal M)\cong  F(\mathcal M' )@>F(p_2)>> F(\mathcal M) \\
@VF(p_1)VV @VF(g)VV \\
F(\mathcal M) @>F(g)>> F(\mathcal N) \\
\end{CD}
\end{equation} Since $g$ is an epimorphism in the abelian category $QCoh(X)$, the cartesian
square on the left is also cocartesian. Then, $F$ being cocontinuous, the right hand square 
in \eqref{4.24e} is also cocartesian and hence $F(g)$ is an isomorphism. 

\medskip
Finally, since any morphism $g:\mathcal M\longrightarrow \mathcal N$ can be factorized
as an epimorphism followed by a monomorphism, it follows from the above that
if $g|_U$ is an isomorphism, $F(g)$ is an isomorphism in $K-Mod$. From the result recalled in Proposition \ref{P4.6fp}(a), 
  it follows that $F:QCoh(X)\longrightarrow K-Mod$ is $i$-local. 

\end{proof} 

\medskip
\begin{Thm}\label{P5.13wvo} Let $X$ be a quasi-compact, semi-separated and Noetherian scheme over $(\mathcal C,\otimes,1)$. 
Let $K$ be a field object of $(\mathcal C,\otimes,1)$ and $F:QCoh(X)\longrightarrow K-Mod$ be a cocontinuous
symmetric monoidal functor that is also normal. Then, there exists a morphism $f:Spec(K)\longrightarrow X$  such that 
$F\cong f^*$. 
\end{Thm} 

\begin{proof} From Proposition \ref{P5.12wvo}, we know that there exists $(i:U=Spec(A)\longrightarrow X)\in ZarAff(X)$ such that
$F$ is $i$-local. Since $F$ is also normal, it follows from
Definition \ref{NDef} that there is a normal functor $(G^U,\tau^U):A-Mod\longrightarrow K-Mod$ with 
$F\cong G^U\circ i^*$. Then $G^U\cong G^U\circ i^*\circ i_*\cong F\circ i_*$. From Proposition \ref{P4.6fp}(b),
we know that  
\begin{equation}
Fun_{c\otimes}(QCoh(U),K-Mod) \doublerightleftarrow{\circ i^*}{\circ i_*} \{F'\in Fun_{c\otimes}(QCoh(X),K-Mod), \textrm{ }\mbox{$i$-local}\}
\end{equation}  is an equivalence of categories. Therefore, we have
 $G^U\in Fun_{c\otimes}(QCoh(U),K-Mod)$.  Further, since $A$ is Noetherian, it follows from Proposition \ref{TYpcf} that there exists a morphism
$g:A\longrightarrow K$ in $Comm(\mathcal C)$ such that $G^U(M)=M\otimes_AK$ for any 
$M\in A-Mod$. If we continue to denote by $g$ the opposite morphism 
$g:Spec(K)\longrightarrow Spec(A)$, we see that $G^U=g^*:QCoh(Spec(A))=A-Mod\longrightarrow K-Mod=QCoh(Spec(K))$.  Hence, $F\cong G^U\circ i^*=g^*\circ i^*\cong (i\circ g)^*$. 

\end{proof}

\medskip
Conversely, given a morphism $f:Spec(K)\longrightarrow X$, it is clear that the pullback functor $f^*:QCoh(X)
\longrightarrow K-Mod$ is cocontinuous and symmetric monoidal. We conclude by showing that the functor $f^*$ is also normal.

\medskip
\begin{Thm}\label{finres} Let $X$ be a quasi-compact, semi-separated and Noetherian scheme over $(\mathcal C,\otimes,1)$. 
Let $K$ be a field object of $(\mathcal C,\otimes,1)$ and let $f:Spec(K)\longrightarrow X$ be a morphism of schemes.
Then, the functor $f^*:QCoh(X)\longrightarrow K-Mod$ is a cocontinuous, symmetric monoidal and normal functor.
\end{Thm}

\begin{proof}  We have mentioned before that $f^*$ is cocontinuous and symmetric monoidal. Let $(i:U=Spec(A)
\longrightarrow X)\in ZarAff(X)$ be such that $f^*$ is $U$-local. We now consider the fiber square:
\begin{equation}\label{5.26eq}
\begin{CD}
Y @>g>> U=Spec(A) \\
@VhVV @ViVV \\
Spec(K) @>f>> X \\
\end{CD}
\end{equation} In \eqref{5.26eq}, $h:Y\longrightarrow Spec(K)$ is a Zariski open immersion. Since $K$ is a field
object, it follows from Proposition \ref{P4.6dj} that either $Y=Spec(0)$ or $h:Y\longrightarrow Spec(K)$ is 
an isomorphism. 

\medskip
\emph{Case 1: $Y=Spec(0)$:} Let $\mathcal O_X$ (resp. $\mathcal O_U$) be the structure sheaf of 
$X$ (resp. $U$) as defined in \eqref{2.3pje}.  Then, since $f^*$ is $U$-local and the morphism $\mathcal O_X\longrightarrow i_*\mathcal O_U$
is an isomorphism when restricted to $U$, $f^*(\mathcal O_X\longrightarrow i_*\mathcal O_U)$ must be an isomorphism 
in $K-Mod$. Then, $f^*i_*\mathcal O_U\cong f^*\mathcal O_X\cong K$.  We now note that
the family
\begin{equation}
\{W\times_XSpec(K)\longrightarrow Spec(K)| W\in ZarAff(X)\} 
\end{equation} is a covering of $Spec(K)$. Hence, we can choose $W\in ZarAff(X)$ such that $W\times_XSpec(K)\ne Spec(0)$. Then, since $K$ is a field object, we must have $Z:=W\times_XSpec(K)\overset{\cong}{\longrightarrow}Spec(K)$. 
Then, the pullback $f^*(i_*\mathcal O_U)$ of $i_*\mathcal O_U$ can be described as:
\begin{equation}\label{5.28EQ}
f^*(i_*\mathcal O_U)\cong (i_*\mathcal O_U)(W)\otimes_{\mathcal O_X(W)}K \cong \mathcal O_X(W\times_XU)\otimes_{\mathcal O_X(W)}K
\end{equation} On the other hand, we notice that:
\begin{equation}\label{5.29EQ}
\begin{array}{ll}
(W\times_XU)\times_WZ&=(W\times_XU)\times_W(W\times_XSpec(K))\\
&=(U\times_XSpec(K))\times_{Spec(K)}(W\times_X
Spec(K))\\ 
&=Spec(0)\times_{Spec(K)}Spec(K)=Spec(0)\\
\end{array}
\end{equation} Now since $Z=Spec(K)$, combining \eqref{5.28EQ} and \eqref{5.29EQ}, we see that 
\begin{equation}
f^*(i_*\mathcal O_U)\cong \mathcal O_X(W\times_XU)\otimes_{\mathcal O_X(W)}K=0
\end{equation} which contradicts the fact that $f^*i_*\mathcal O_U\cong K$. Hence, it is not possible to have
$Y=U\times_XSpec(K)=Spec(0)$.

\medskip
\emph{Case 2: $h:Y\longrightarrow Spec(K)$ is an isomorphism:} We may suppose that $h$ is actually
the identity. Then, from \eqref{5.26eq}, we see that $f=i\circ g$ and hence $f^*\cong g^*\circ i^*$. Since $g:Y=Spec(K)
\longrightarrow Spec(A)$ corresponds to a morphism in $Comm(\mathcal C)$, it is clear that $g$ induces a normal functor
$(g^*,\tau):A-Mod\longrightarrow K-Mod$.  Hence, the result follows.

\end{proof}

\end{document}